\documentclass[11pt, twoside]{amsart}

\title[]{Thompson's Group $V$ and Virtual Link Theory}

\author[M. Chrisman]{Micah Chrisman}
\address{Chrisman: Department of Mathematics, The Ohio State University at Marion, 1465 Mount Vernon Avenue, Marion, OH 43302}
\email{chrisman.76@osu.edu}

\author[L. Liles]{Louisa Liles}
\address{Liles: Department of Mathematics, Oberlin College, 10 N. Professor Street, Oberlin, OH 44074}
\email{lliles@oberlin.edu}

\author[M. Molander]{Melody Molander}
\address{Molander: Department of Mathematics, The Ohio State University, 100 Math Tower, 231 W 18th Ave, Columbus, OH 43210}
\email{molander.3@osu.edu}

\usepackage{amsmath}
\usepackage{amsfonts}
\usepackage{amsthm}
\usepackage{svg}
\usepackage{amssymb,bbm}
\usepackage[mathscr]{euscript}
\usepackage{dsfont}
\usepackage{stmaryrd}
\usepackage{appendix}
\usepackage[english]{babel}
\usepackage{url}
\usepackage{fancyhdr}
\usepackage{graphicx}
\usepackage{verbatim}
\usepackage[normalem]{ulem}
\usepackage[shortlabels]{enumitem}
\newcommand{\stkout}[1]{\ifmmode\text{\sout{\ensuremath{#1}}}\else\sout{#1}\fi}
\usepackage[
colorlinks=true,
linkcolor=black, 
anchorcolor=black,
citecolor=black,
urlcolor=black, 
]{hyperref}
\usepackage[all,cmtip]{xy}
\usepackage{geometry}
\usepackage[dvipsnames]{xcolor}
\usepackage{lscape} 
\usepackage{multicol}
\usepackage{shortcuts}
\usepackage[utf8]{inputenc}
\usepackage{tikz}
\usepackage{tikz-cd}
\usepackage{overpic}

\allowdisplaybreaks

\usepackage{adjustbox}

\usepackage{enumitem}

\usepackage[notcite,notref,final]{showkeys}

\usepackage{calrsfs}
\DeclareMathAlphabet{\cal}{OMS}{zplm}{m}{n}

\DeclareMathAlphabet{\mathsf}{OT1}{cmss}{m}{n} % to math mathsf thinner

\usepackage{relsize}

\makeatletter
\DeclareRobustCommand\widecheck[1]{{\mathpalette\@widecheck{#1}}}
\def\@widecheck#1#2{%
    \setbox\z@\hbox{\m@th$#1#2$}%
    \setbox\tw@\hbox{\m@th$#1%
       \widehat{%
          \vrule\@width\z@\@height\ht\z@
          \vrule\@height\z@\@width\wd\z@}$}%
    \dp\tw@-\ht\z@
    \@tempdima\ht\z@ \advance\@tempdima2\ht\tw@ \divide\@tempdima\thr@@
    \setbox\tw@\hbox{%
       \raise\@tempdima\hbox{\scalebox{1}[-1]{\lower\@tempdima\box
\tw@}}}%
    {\ooalign{\box\tw@ \cr \box\z@}}}
\makeatother

\numberwithin{equation}{section}

\newtheorem{theorem}{Theorem}[subsection]

\newtheorem{proposition}[theorem]{Proposition}

\newtheorem{corollary}[theorem]{Corollary}
\newtheorem{lemma}[theorem]{Lemma}

\newtheorem{theorem*}{Theorem}
\newtheorem{thm}{Theorem}

\theoremstyle{definition}
\newtheorem{definition}[theorem]{Definition}

\newtheorem{remark}[theorem]{Remark}

\newtheorem{question}{Question}[section]

\makeatletter              % This sequence of commands will
\let\c@equation\c@theorem  % incorporate equation numbering
\makeatother
\numberwithin{equation}{section}

\makeatletter
\@namedef{subjclassname@2020}{%
  \textup{2020} Mathematics Subject Classification}
\makeatother


\tikzcdset{scale cd/.style={every label/.append style={scale=#1},
    cells={nodes={scale=#1}}}}

\begin{document}
\begin{abstract} 
Thompson’s groups $F \subset T \subset V$ were introduced in 1965 and have since found widespread application in fields as diverse as logic, group theory, homotopy theory, and lattice gauge theory. In 2014, V. F. R. Jones constructed unitary representations of $F$, factoring through a surjection from $F$ to isotopy classes of links in $S^3$. The second author extended Jones’ surjection to $T$, thereby constructing all isotopy classes of checkerboard colorable (CC) links in the thickened annulus. We complete this program for $V$, defining a surjection $\mathscr{L}_{V}$ from $V$ to virtual equivalence classes of CC links in thickened compact oriented surfaces. This yields a new oriented subgroup $\vec{V} \subset V$ containing Jones' oriented subgroups $\vec{F} \subset F$ and $\vec{T}\subset T$. We prove $\vec{V}$ realizes all oriented almost classical virtual links. We then construct unitary representations of $V$ and $\vec{V}$ from kei and operator quandle coloring invariants, respectively. 
\end{abstract}

\maketitle

\section{Introduction and Statement of Main Results}

Thompson's groups $F \subset T \subset V$ are certain groups of piecewise-linear maps on the interval $[0,1]$ and the circle $S^1$. In \cite{jones_unitary}, Jones suggested using unitary representations of Thompson's groups to build toy models of conformal field theories on $S^1$, where $F$ and $T$ are viewed as rough approximations to $\text{Diff}^+(S^1)$. Using a planar algebra of Conway tangles, Jones constructed a surjective map from $F$ to the set of isotopy classes of links in the $3$-sphere. This method yields unitary representations of $F$ and $T$. The program was further extended by the second author, who constructed a map from $T$ to links in the thickened annulus \cite{LL25}. Although $V$ has long been of interest in group theory (cf. Belk et al \cite{BBQS}) and its theory of unitary representations has been recently studied by Brothier and Jones \cite{brothier_jones_a,brothier_jones_b,brothier_classification}, $V$ has not yet received a knot-theoretic interpretation. Here we complete this geometric part of Jones' program for $F\subset T \subset V$.

Elements of $F$ yield links in $S^3$ due to their representation by a pair of \emph{planar} rooted bifurcating trees $(T_+,T_-)$. Elements of $T$ allow a cyclic permutation of the leaves of these trees, so that the corresponding links are annular. Elements of $V$, however, permit an arbitrary permutation of the leaves. Their representing diagrams are thus non-planar. Here we solve the non-planarity problem by restating it in terms of Kauffman's \emph{virtual link theory} \cite{kauffman_vkt}. This is a generalization of classical link theory that admits non-planar diagrams. Such diagrams can be realized as links $\mathcal{L} \subset \Sigma \times [0,1]$, where $\Sigma$ is a compact orientable surface. If for some representative $\mathcal{L}$ of a virtual link type $L$, $[\mathcal{L}]=0 \in H_1(\Sigma \times [0,1];\mathbb{Z}_2)$, then $L$ is called \emph{checkerboard-colorable (CC)}. We prove:

\begin{thm} \label{thm_intro_real} Every CC virtual link type $L$ is realized by an element of Thompson's group $V$. \end{thm} 

Note that the realization algorithms for $F$ and $T$, from \cite{jones_unitary} and \cite{LL25}, respectively, depend on an initial checkerboard coloring of the link diagram. In the case of $F$, the checkerboard coloring is for diagrams on the $2$-sphere $S^2$ and for $T$, it is for diagrams on the annulus $S^1 \times [0,1]$. Likewise, our map $\mathscr{L}_V$ from $V$ to virtual link diagrams produces \emph{only} CC virtual link diagrams and, by Theorem \ref{thm_intro_real}, \emph{all} CC virtual link types are in the image of $\mathscr{L}_V$. Furthermore, just as Jones' realization theorem of links in $S^3$ gave rise to new ``oriented" subgroups $\vec{F} \leq F$ and $\vec{T} \leq T$ \cite{jones_unitary}, Theorem \ref{thm_intro_real} allows for the identification of an oriented subgroup $\vec{V} \leq V$. Elements of $\vec{V}$ map to oriented virtual links with orientation induced by an oriented spanning surface. Consequently, virtual links arising from $\vec{V}$ are \emph{almost classical (AC)} in the sense of Silver-Williams \cite{silver_williams}. Conversely, we prove a generalization of Aiello's theorem for $\vec{F}$ \cite{aiello_oriented}:

\begin{thm} \label{thm_intro_oriented} Every oriented AC virtual link type is realized by an element of the subgroup $\vec{V} \subset V$. 
\end{thm}

Jones' technology produces unitary representations for Thompson's groups $F \subset T \subset V$, and more generally, for categories admitting a certain group of fractions \cite{jones_no_go}. As shown in \cite{jones_no_go}, this does not succeed in producing conformal field theories from planar algebras as originally intended. However, more recent work of Brothier-Stottmeister  \cite{brothier_stottmeister_1,brothier_stottmeister_2}, has shown that Jones' program instead has intersections with lattice gauge theory, loop quantum gravity, and Yang-Mills theory. This has renewed a desire to understand unitary representations of $F$, $T$, and $V$, with the group $V$ being of particular interest (see, e.g. \cite{brothier_jones_a,brothier_jones_b, brothier_survey,brothier_wreath,brothier_classification}). 

Using Theorems \ref{thm_intro_real} and \ref{thm_intro_oriented}, we construct unitary representations of $\vec{V}$ and $V$ by composing with virtual link invariants. This result extends previous work of Aiello, Conti, and Jones \cite{jones_unitary,AC19,ac19oriented, aiellocontijones}, who establish invariants of links in $S^3$ as positive definite functions on $F$ and $T$. Via the Gelfand–Naimark–Segal correspondence between positive definite functions and unitary representations \cite{GN43, Seg47}, it follows that these invariants can be recovered from unitary representations of $F$ and $T$. For Thompson's group $V$, we focus on an infinite family of virtual link invariants called operator quandle colorings (see e.g. \cite{KMV25}). Positive definite functions of $V$ and $\vec{V}$ are then defined by counting the operator quandle colorings of a virtual link type: 

\begin{thm}\label{intro-thm-operator-quandle}
Let $(Q,*)$ be a finite quandle.
\begin{enumerate}
    \item If $(Q,*)$ is an involutive quandle, then there is a quandle coloring function $V\to \mathbb{C}$ which is positive definite on $V$.
    \item For every automorphism of $(Q,*)$, there is an operator quandle coloring function $\vec{V}\to \mathbb{C}$ which is positive definite on $\vec{V}$.
\end{enumerate}
Consequently, each of these virtual link invariants can be recovered from a unitary representation of the corresponding group. 
\end{thm}

Since dihedral quandles are involutive and Fox $n$-colorings are dihedral quandle coloring invariants, the positive definite functions in Theorem \ref{intro-thm-operator-quandle} extend those derived from classical Fox $n$-colorings by Aiello and Conti \cite{AC19}. By Gra\~{n}a \cite{grana}, all quandle coloring invariants for virtual links are quantum invariants. Theorem \ref{intro-thm-operator-quandle} thus shows that a large class of quantum invariants of virtual links yield unitary representations of $V$ and its oriented subgroup $\vec{V}$.

This project was partly inspired by a question raised in \cite{KTVF}. There, Kodama and Takano used the theory of diagram groups to construct a \textit{virtual Thompson group} $\mathit{VF}$, which contains $F$ as a subgroup and gives rise to all virtual links. In \cite{KTVF}, the authors ask how $\mathit{VF}$ relates to other generalizations of $F$. On the level of diagrams, the elements of $\mathit{VF}$ are represented by pairs of \emph{planar} rooted bifurcating trees $(S_+,S_-)$ where the edges are labeled as either classical or virtual. Hence, in the case of $\mathit{VF}$, the non-planarity problem does not arise. In the present paper, however, the flexibility afforded by non-planar virtual link diagrams is their fundamental feature, since it allows us to directly represent the inherent non-planarity of Thompson's group $V$. Consequently, we obtain unitary representations of $V$ from virtual link invariants. By introducing a common target space of virtual links, we hope that the relationship between $\mathit{VF}$ and $V$ can be clarified with future investigation. More broadly, our construction raises natural questions about how Thompson-like groups encode the geometric properties of virtual links.

Our paper is organized as follows. Section \ref{sec: background} introduces Thompson's groups and previous methods of building links in $S^3$ and annular links from elements of $F$ and $T$, respectively. For the reader's convenience, a review of virtual links, CC virtual links, and virtual Tait graphs is given in Section \ref{sec_virtual_links} and \ref{sec_virt_tait_graphs}. Section \ref{sec: construction} defines our map from $V$ to the set of checkerboard colorable virtual link types. Theorem \ref{thm_intro_real} is proved in Section \ref{sec_real}. An example calculation is given in Section \ref{sec_example}. This illustrates the realization algorithm for the vertical mirror image of the CC virtual knot 4.105 from Green's table \cite{green}. A comparison with $\mathit{VF}$ is then given in Section \ref{sec_example}. The oriented subgroup is defined in Section \ref{sec: oriented}, which also contains the proof of Theorem \ref{thm_intro_oriented}. Section \ref{sec: op colorings} gives a brief review of quandles, operator quandle colorings, and positive definite functions. The proof of Theorem \ref{intro-thm-operator-quandle} appears in Section \ref{sec_proof_of_C}. Section \ref{sec: openq} explores future directions and related open questions. 

\section*{Acknowledgments}
LL and MM were each supported in part by an AMS-Simons Travel Grant. All of the authors would like to thank S. V. Chmutov for helpful conversations about Thompson's groups and links. 

\section{Background}\label{sec: background}

Section \ref{sec_F_T_V} reviews the groups $F\subset T \subset V$. In Section \ref{sec_strand}, we review strand diagrams and prove some lemmas about diagrammatic composition in $V$ which will be needed ahead in Section \ref{sec: oriented}. Virtual links, checkerboard colorable virtual links, and almost classical links are reviewed in Section \ref{sec_virtual_links}. Finally, in Section \ref{sec_virt_tait_graphs}, we review virtual spanning surfaces and virtual Tait graphs.

\subsection{Thompson's Groups $F \subset T \subset V$} \label{sec_F_T_V}
Elements of $F$ are piecewise linear self-homeomorphisms of $[0,1]$ such that all points of non-differentiability are dyadic rational numbers (i.e. of the form $\frac{a}{2^b},a, b\in \mathbb{Z}_{\geq 0}$) and all derivatives are powers of $2$.  For example, \[g(t)=\begin{cases} \frac{1}{2}t & 0 \leq t \leq \frac{1}{2}\\ t-\frac{1}{4} & \frac{1}{2} \leq t \leq \frac{3}{4} \\ 2t-1 & \frac{3}{4} \leq t \leq 1\end{cases}\] is an element of $F$. Points of non-differentiability are $\frac{1}{2}$ and $\frac{3}{4},$ and the derivatives are $2^{-1}, 2^0,$ and $2^1$. The group operation in $F$ is composition of functions.

Elements of $F$ can be conveniently encoded as ordered pairs of planar binary trees with the same number of leaves. The details of this correspondence can be found in \cite[Section 2]{cfpnotes}. Informally, the first tree in the ordered pair corresponds to a partition of the domain, and the second tree corresponds to the image of that partition. For example, the pair of trees in Figure \ref{fig: F trees} encodes the map $g$ defined above. The first tree comes from the partition \(\{[0, \frac{1}{2}] \cup [\frac{1}{2}, \frac{3}{4}]\cup [\frac{3}{4},1]\}\) and the second tree comes from the partition \(\{[0,\frac{1}{4}]\cup [\frac{1}{4},\frac{1}{2}] \cup [\frac{1}{2},1]\}\), where each interval is the image of an interval in the first partition, e.g. $g([0, \frac{1}{2}])=[0,\frac{1}{4}].$ As a shorthand for depicting an ordered pair of trees, the second tree is often reflected over the horizontal axis and attached to the first tree at the leaves, as in Figure \ref{fig: F trees}. We will use this notation throughout the rest of the paper.  

\begin{figure}
\includegraphics[scale=0.3]{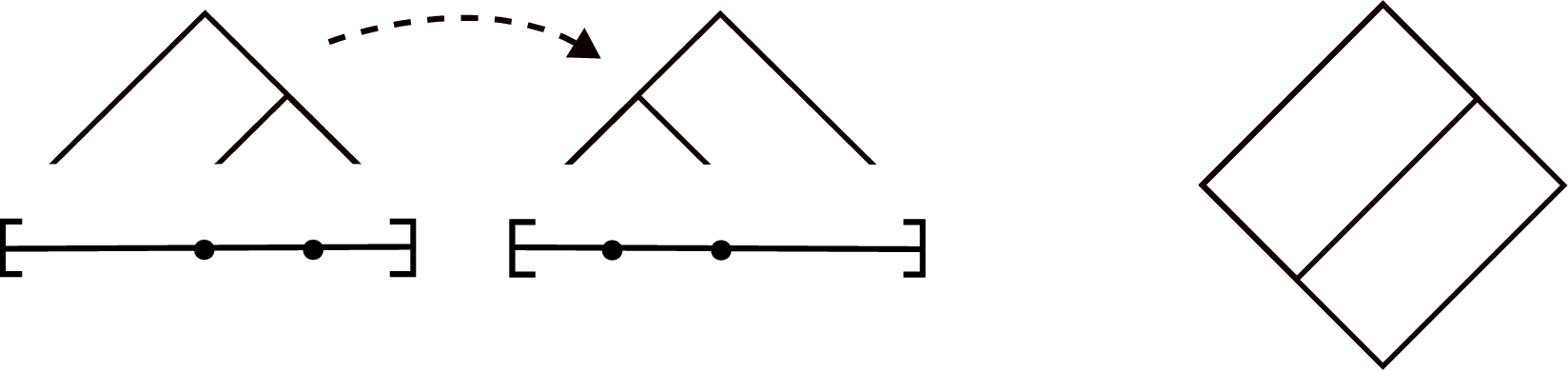}\put(-170,60){$g$}\put(-210,-3){$\frac{1}{2}$}\put(-193,-3){$\frac{3}{4}$}\put(-150,-3){$\frac{1}{4}$}\put(-132,-3){$\frac{1}{2}$}
\caption{Encoding $g \in F$ as a pair of trees.}\label{fig: F trees}\end{figure}
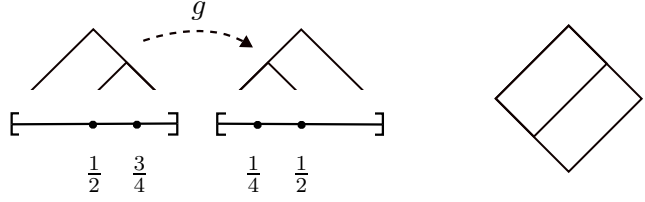

Any ordered pair of planar, rooted, binary trees with the same number of leaves determines an element of $F$, but two such pairs may represent the same function if they differ by \textit{canceling carets}, see Figure \ref{fig: F carets}. On the level of functions, canceling carets correspond to unnecessarily refining partitions. For example, rewriting $g$ as \[g(t)=\begin{cases} \frac{1}{2}t & 0 \leq t \leq \frac{1}{4}\\ \frac{1}{2}t & \frac{1}{4} \leq t \leq \frac{1}{2}\\ t-\frac{1}{4} & \frac{1}{2} \leq t \leq \frac{3}{4} \\ 2t-1 & \frac{3}{4} \leq t \leq 1\end{cases}\] gives the pair of trees shown above. A pair of trees with no canceling carets is called \textit{reduced}.

\begin{figure}
\includegraphics[scale=0.3]{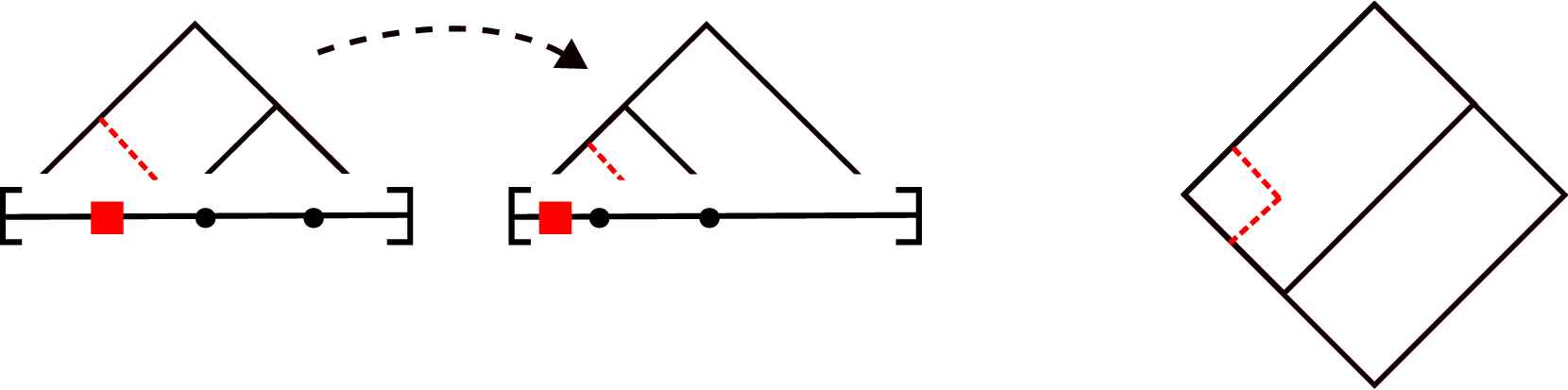}\put(-175,60){$g$}\put(-212,0){$\frac{1}{2}$}\put(-225,0){$\frac{1}{4}$}\put(-195,0){$\frac{3}{4}$}\put(-160,0){$\frac{1}{8}$}\put(-150,0){$\frac{1}{4}$}\put(-134,0){$\frac{1}{2}$}
\caption{A canceling caret in Thompson's group $F$.}\label{fig: F carets}\end{figure}
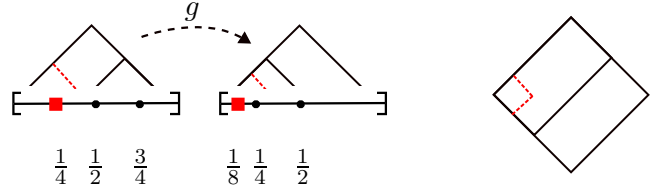

Thompson's group $F$ is a subgroup of the larger group $T$, whose elements are piecewise linear self-homeomorphisms of $[0,1]/0 \sim 1$, subject to the same conditions on differentiability and derivatives as elements of $F$. For example, \[h(t)=\begin{cases} \frac{1}{2}t+\frac{1}{4}& 0 \leq t \leq \frac{1}{2}\\ 2t-\frac{1}{2} & \frac{1}{2} \leq t \leq \frac{3}{4}\\ t-\frac{3}{4} & \frac{3}{4} \leq t \leq 1\end{cases}\] is an element of $T$ but not $F$. A function $g \in T$ still has an associated ordered pair of trees, but this pair alone is not enough to specify the function, as the leaves may be attached according to any cyclic permutation of the $n$ leaves. Moreover, this permutation is completely determined by where one leaf is sent. If we use the positive integer $k \leq n$ to specify that the first leaf of the top tree is sent to the $k$-th leaf of the bottom tree, then every element of $T$ can be encoded as a triple $(T_+,T_-;k)$ where $T_+$ and $T_-$ are trees. For example, the element $h$ above has $n=3,k=2$. We indicate $k$ graphically by adding a decoration to the $k$th leaf of $T_-$, as in Figure \ref{fig: triple t}. Canceling carets (and therefore unreduced triples) still occur in $T$ if partitions are refined unnecessarily, however they become less diagrammatically obvious, now that the leaves are not always attached from left to right. The subgroup $F$ is the set of elements for which $k=1$.
\begin{figure}
    \includegraphics[scale=0.4]{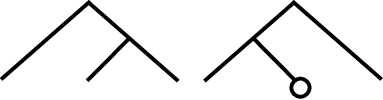}\caption{The decorated pair of trees belonging to $h(t) \in T$. Here, $k=2$.}\label{fig: triple t}
\end{figure}

Thompson's group $V$ contains $T$ as a subgroup. Elements of $V$ are right-continuous piecewise linear bijections of $[0,1]$, subject to the usual conditions on differentiability and derivatives. Relaxing the assumption about continuity means that a function in $V$ may now map intervals according to any permutation. For example,

\[k(t)=\begin{cases} t+\frac{1}{2} & 0 \leq t < \frac{1}{2} \\ t-\frac{1}{4} & \frac{1}{2} \leq t < \frac{3}{4} \\  t - \frac{3}{4} & \frac{3}{4} \leq t \leq 1\end{cases}\] is in $V$ but not $T$.
    Elements of $V$ can be expressed as an ordered pair of planar, binary trees with $n$ leaves, together with a permutation $\sigma$ in $S_n$, the symmetric group on $n$ letters, which specifies how the trees are attached. The pair of trees on the left of Figure \ref{fig: V strand diag} give an example of a triple $(T_+,T_-,\sigma)$ associated to $k(t)$, and the permutation $\sigma$ is depicted by labeling the leaves of the second tree (see the left side of Figure \ref{fig: V strand diag}). As was the case in $T$, it is still possible for a triple $(T_+,T_-,\sigma)$ to be unreduced, but canceling carets are less visually obvious now that the leaves of the trees may be attached according to any permutation. 

    \begin{figure}\includegraphics[scale=0.5]{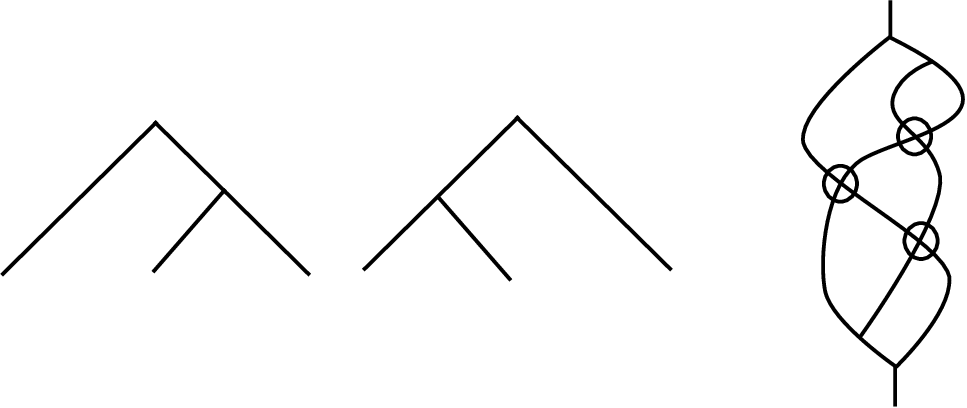}\put(-70,15){$1$}\put(-110,15){$2$}\put(-145,15){$3$}\caption{A triple $(T_+,T_-,\sigma) \in V$ and its \textit{abstract strand diagram} (see Definition \ref{def: abstract}).}\label{fig: V strand diag}
\end{figure}

\subsection{Strand diagrams}\label{sec_strand}

We have just described diagrammatic representations of elements in $F, T,$ and $V$. To encode the group operation, Belk introduced \textit{strand diagrams} \cite[Section 7.1]{belk_thesis}.

\begin{definition}\cite[Definition 2.6]{belkmatucci} An $(m,n)$\textit{-strand diagram} is a trivalent directed graph embedded in the unit square with $m$ univalent sources on the top edge and $n$ univalent sinks on the bottom edge. All edges have nonzero slopes and all interior vertices are either \textit{merges} or \textit{splits}.
\end{definition}

\begin{figure}\includegraphics[scale=0.3]{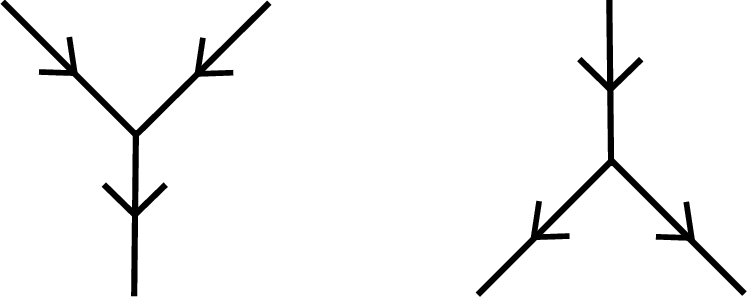}\caption{A merge (left) and  split(right).}\label{fig: merge split}\end{figure}
A strand diagram can be thought of like a braid, but instead of twists there are merges and splits, and the number of top endpoints is not necessarily equal to the number of bottom endpoints. If the edges of a strand diagram do not have arrows indicating a direction, we assume all edges to be given the downward orientation. Examples are shown in Figure \ref{fig: sd ex}. We consider strand diagrams up to planar isotopy and \textit{reductions}. 

    \begin{figure}\includegraphics[scale=0.5]{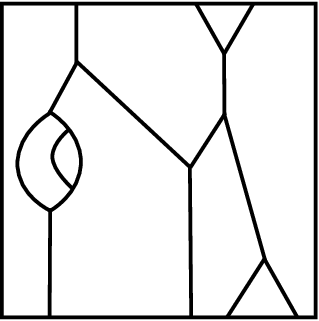}\qquad \includegraphics[scale=0.5]{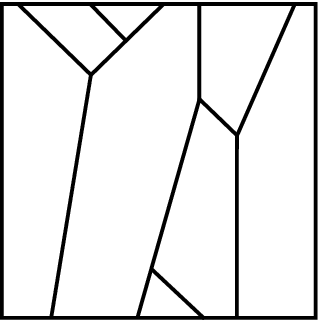}\caption{A$(3,4)$-strand diagram and a $(5,4)$-strand diagram. All edges are assumed to be oriented in the downward direction. The diagram on the right is reduced in the sense of Definition \ref{def: reduced} below, but the diagram on the left is not. }\label{fig: sd ex}\end{figure}

\begin{definition}\label{def: reduced} [\cite{belkmatucci}, Def. 2.1]
A reduction of a strand diagram is a move of Type I or II as depicted in Figure \ref{fig: strand diag moves}. A strand diagram is called \textit{reduced} if it is not subject to any reductions.
\end{definition}
\begin{figure}\includegraphics[scale=0.5]{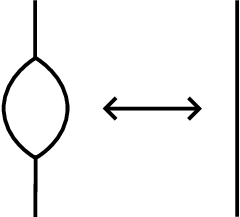}\qquad \includegraphics[scale=0.5]{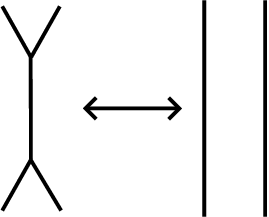}\caption{Type I (left) and Type II (right) moves defined in \cite{belkmatucci}.}\label{fig: strand diag moves}\end{figure}

Every strand diagram is equivalent to a unique reduced strand diagram, and the set of reduced $(1,1)$-strand diagrams, with group operation of concatenation and reduction, forms a group $\mathscr{R}.$ 
\begin{proposition}[\cite{belkmatucci}, Prop 2.5]
The group $\mathscr{R}$ is isomorphic to Thompson's group $F$.
\end{proposition}

Given an element $g \in F$, its corresponding element of $\mathscr{R}$ is its reduced pair of trees, each with an additional root added at the top, with each pair of attached leaves smoothed together as one edge. The root of the top tree serves as the univalent source and the root of the bottom tree is the univalent sink. To calculate $g \circ f$ diagrammatically, one can stack the reduced pair of trees for $g$ below the reduced pair of trees for $f$ to obtain an unreduced $(1,1)$-strand diagram. Reducing with moves of Type I and II will give the reduced pair of trees for $g \circ f$.

Belk and Matucci extend strand diagrams to Thompson's group $T$ by defining \textit{cylindrical strand diagrams}, see \cite[Definition 2.7]{belkmatucci}. Every cylindrical strand diagram is equivalent to a unique \textit{reduced} cylindrical strand diagram, and the group of reduced $(1,1)$-cylindrical strand diagrams, with the operation of concatenation and reduction, is isomorphic to $T$.

To encode the group $V$, Belk and Matucci defined \textit{abstract strand diagrams.}

\begin{definition}[\cite{belkmatucci}, Def. 2.8]\label{def: abstract} An \textit{abstract $(m,n)$-strand diagram} is a finite, acyclic (not necessary planar) graph with $m$ univalent sources, $n$ univalent sinks, all other vertices are either splits or merges, each equipped with a cyclic ordering of the three edges that meet. The collection of these cyclic orderings is called a \textit{rotation system.} Abstract strand diagrams are considered equivalent up to reductions and directed graph isomorphisms, provided that both the reductions and the isomorphisms preserve the rotation system; other reductions and isomorphisms are not permitted. The set of reduced $(1,1)$-abstract strand diagrams, with group operation of concatenation and reduction, is isomorphic to $V$. 
\end{definition}

\begin{remark} Because abstract strand diagrams are equivalent up to rotation-system-preserving directed graph isomorphisms, we may change how the graph is embedded in the plane \textit{without changing the strand diagram} as long as we keep the same cyclic ordering of edges at each vertex. This differs from the cases of non-abstract or cylindrical strand diagrams, in which the embedding of the graph may only be modified by isotopy.
\end{remark}

Given a reduced triple $(T_+,T_-,\sigma)$, we can build an abstract strand diagram by reflecting $T_-$ over the horizontal axis and attaching the leaves of $T_+$ according to $\sigma$, see Figure \ref{fig: V strand diag}. To distinguish vertices from edges simply passing over each other in a non-planar way, we will denote the latter with a circle drawn around the crossing, as in Figure \ref{fig: V strand diag}.

\begin{remark}
    When creating an abstract strand diagram for $(T_+,T_-,\sigma)$ one has infinitely many choices of how to depict $\sigma$. By Definition \ref{def: abstract}, all choices lead to equivalent diagrams because they are all related by rotation-system preserving graph isomorphisms. On the level of virtual links, this equivalence will correspond to detour moves between diagrams (see Section \ref{sec_virtual_links} ahead).
\end{remark}
 Given abstract strand diagrams $D_g$, $D_f$ for $g, f\in V$, one can obtain a diagram for $g \circ f$ by stacking $D_g$ below that of $D_f$, and then reducing the result. Although any sequence of reductions will lead to an equivalent diagram, it will be helpful for our purposes to use a particularly convenient sequence of moves. This procedure always results in a diagram that clearly depicts the triple $(V_+,V_-,\tau)$ for $g \circ f$. Other sequences of moves will lead to strand diagrams that are equivalent up to a rotation-system-preserving directed graph isomorphisms, but these isomorphisms may visually obscure the associated pair of trees and permutation we are after.

Toward defining this procedure, we first establish some notation. We follow the convention of the literature in which trees and forests that ``grow up" (i.e., contain only merges) are called \textit{inverse trees} and \textit{inverse forests}. We will reserve the use of \textit{trees} and \textit{forests} for refer to diagrams that ``grow down" and contain only splits. Although we have established that elements of $V$ can be thought of as ordered triples consisting of two trees with the same number of leaves, it is helpful at times to think of the second tree as an inverse tree instead. After all, when constructing an abstract strand diagram from $(T_+,T_-,\sigma)$, we first reflect $T_-$ over the horizontal axis, at which point it becomes an inverse tree. We also use $\circ$ to denote when two (potentially unreduced) abstract strand diagrams are stacked but not reduced, and $*$ to denote the operation of stacking and subsequently reducing. To eliminate ambiguity, we will now denote function composition as $gf$, which means that $f$ is applied, followed by $g$. With this notation in mind, we establish some prerequisite lemmas.
\begin{lemma}\label{lem: trees to forests}
    Given a tree $T_+$ and an inverse tree $T_-$, the reduced strand diagram $T_{+} * T_-$ can always be expressed as an inverse forest $F_-$ stacked below a forest $F_+$.
\end{lemma}
\begin{proof}
    The Type II reduction move guarantees no splits occur directly below merges. Therefore all merges may be moved to the bottom half of the diagram, and all splits may be moved to the top half of the diagram, as in Figure \ref{fig: forests from trees}.
\end{proof}

\begin{figure}\includegraphics[scale=0.4]{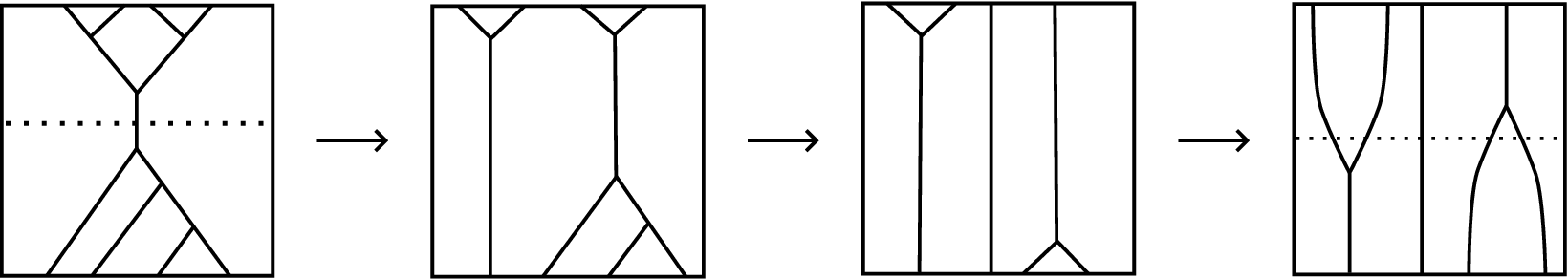}\put(-335,40){$T_-$}\put(-335,10){$T_+$}\put(6,40){$F_{+}$}\put(6,10){$F_-$}\caption{Using Type II moves and planar isotopy to reduce $T_+ \circ T_-$ (left) to $F_- \circ F_+$ (right).}\label{fig: forests from trees}\end{figure}
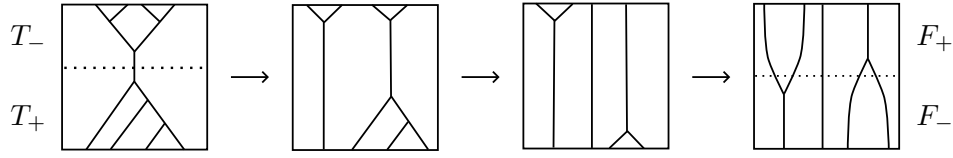
\begin{definition}
    Let $\sigma \in S_n$, the symmetric group on $n$ letters. Let $D(\sigma)$ be an abstract strand diagram (not necessarily planar) illustrating the bijection $\sigma$. The top left square in Figure \ref{fig: slide up} is an example of such a diagram for the permutation $(132) \in S_3$.
\end{definition} 
\begin{lemma}\label{lem: slide up} Consider an abstract strand diagram of the form $F_+ \circ D(\sigma)$, where $F_+$ is a forest and $D(\sigma)$ is as above. Then there exists an equivalent diagram of the form $D(\sigma')\circ F_+^{'},$ where $\sigma' \in S_{n^{'}},$ a potentially different permutation group from $S_{n}$, and $F_+'$ is another forest.\end{lemma}
\begin{proof} Figure \ref{fig: slide up} illustrates the equivalence of abstract strand diagrams. This modification does not change the isomorphism class or the rotation system of the graph. 
\begin{figure}
\includegraphics[scale=0.4]{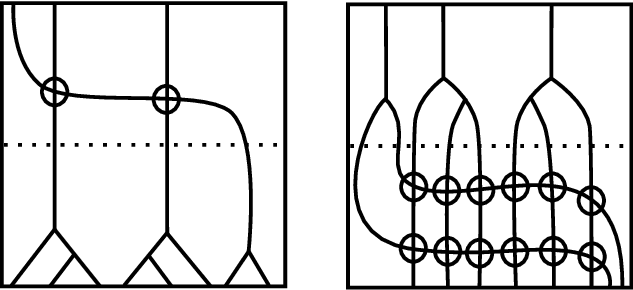}\put(-150,40){$D(\sigma)$}\put(-150,10){$F_+$}\put(6,40){$F_+^{'}$}\put(6,10){$D(\sigma')$}
\caption{The equivalence of abstract strand diagrams in Lemma \ref{lem: slide up}.}\label{fig: slide up}
\end{figure}
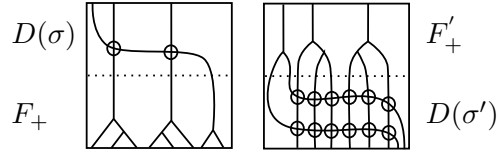
\end{proof} In other words, Lemma \ref{lem: slide up} tells us that forests can ``slide up" through permutations. By a similar argument, inverse forests can ``slide down" through permutations.
\begin{lemma}\label{lem: slide down} Consider an abstract strand diagram of the form $D(\sigma)\circ F_-$, where $F_-$ is an inverse forest. Then there exists equivalent strand diagram of the form $F_-^{'}\circ D(\sigma'),$ where $\sigma'$ is a (possibly different) element of some $S_{n'}$ and $F_-^{'}$ is a (possibly different) inverse forest. \end{lemma}

We may now describe our procedure by which the strand diagrams for $f=(T_+,T_-,\sigma)$ and $ g=(S_+,S_-,\nu)$ can be composed and reduced, resulting in a diagram that depicts $gf=(V_+,V_-,\tau)$. The process below is illustrated in Figure \ref{fig: comp in V}.
\begin{enumerate}
    \item After stacking $S_- \circ D(\nu) \circ S_+ \circ T_- \circ D(\sigma) \circ T_+$, apply Lemma \ref{lem: trees to forests} to $S_+ \circ T_-$ yields the equivalent diagram $S_-\circ D(\nu)\circ (F_- \circ F_+)\circ D(\sigma) \circ T_+$.
    \item Applying Lemma \ref{lem: slide up} to $F_+ \circ D(\sigma)$ and Lemma \ref{lem: slide down} to $D(\nu) \circ F_-$ yields an equivalent diagram of the form $S_- \circ F_-^{'} \circ D(\nu ') \circ D(\sigma') \circ F_+^{'} \circ T_{+}.$ 
    \item The diagrams $(F_+^{'} \circ T_+)$ and $(S_- \circ F_-^{'}$) give a tree and inverse tree with the same number of leaves. The diagram $D(\nu') \circ D(\sigma')$ is equivalent to $D(\nu' \circ \sigma')$ and identifies the leaves of $(F_+^{'} \circ T_+)$ with those of $(S_- \circ F_-^{'})$. Therefore, we have reduced our diagram to a tree $U_+$ and an inverse tree $U_-$ with their leaves identified by the permutation $\nu ' \circ \sigma'.$ All merges appear below $D(\nu' \circ \sigma')$ and all splits appear above it. Therefore no Type II moves can be applied, only Type I moves, which amount to canceling carets. After all carets are canceled, the resulting diagram must be the reduced triple $(V_+,V_-, \tau)$ corresponding to $g \circ f.$
\end{enumerate}

\begin{figure}\includegraphics[scale=0.3]{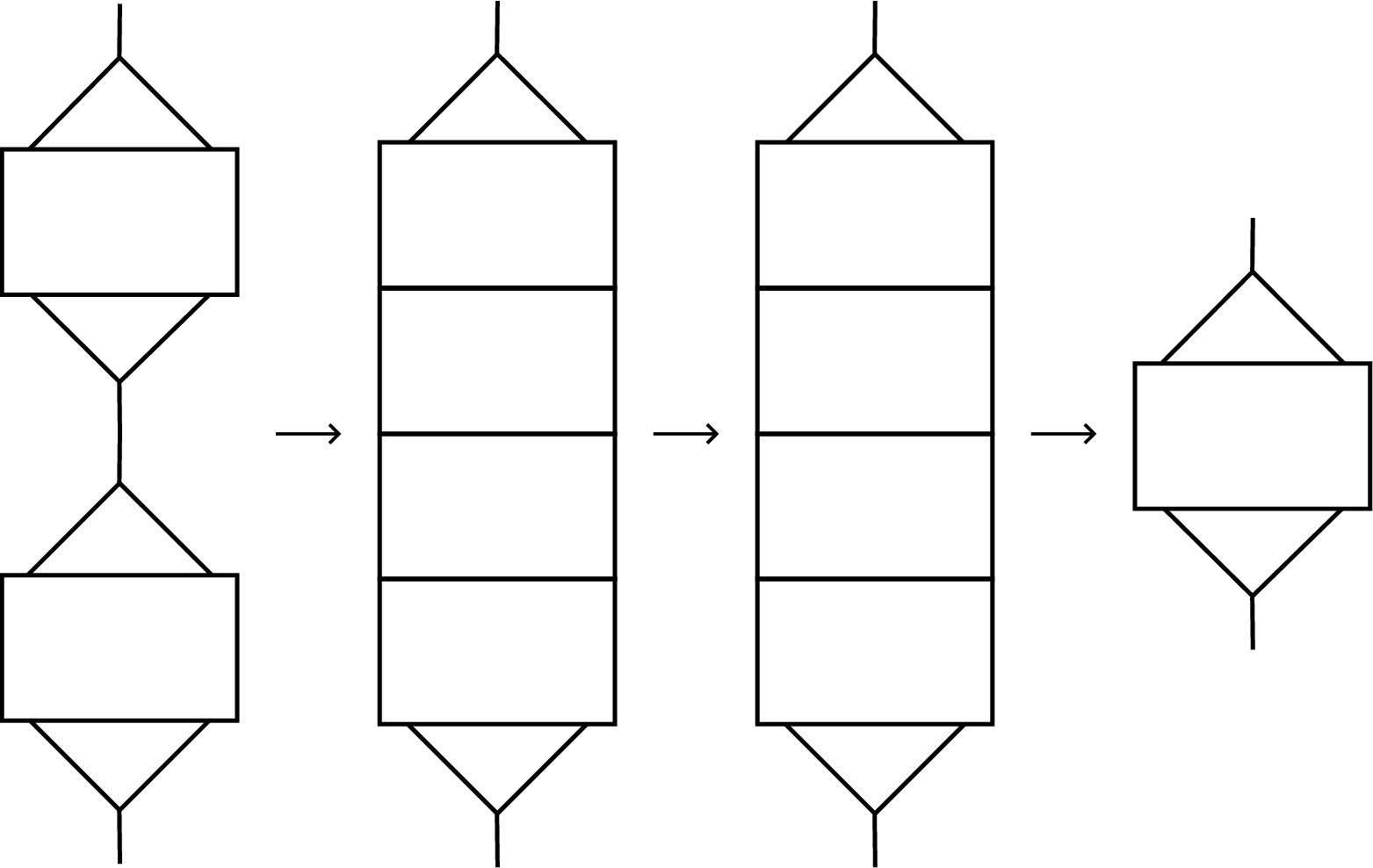}
%rightmost diagram
\put(-38,70){\tiny$D(\nu' \circ \sigma')$}
\put(-24,87){\tiny$U_+$}
\put(-25, 50){\tiny$U_-$}
%third diagram
\put(-87,15){\tiny$S_-$}
\put(-87,33){\tiny$F_-^{'}$}
\put(-90,55){\tiny$D(\nu')$}
\put(-90, 80){\tiny$D(\sigma')$}
\put(-87, 105){\tiny$F_+^{'}$}
\put(-87,122){\tiny$T_+$}
%second diagram
\put(-149,15){\tiny$S_-$}
\put(-152,33){\tiny$D(\nu)$}
\put(-149,55){\tiny$F_-$}
\put(-149,80){\tiny$F_+$}
\put(-152,105){\tiny$D(\sigma)$}
\put(-149,122){\tiny$T_+$}
%left diagram
\put(-210,122){\tiny$T_+$}
\put(-213,105){\tiny$D(\sigma)$}
\put(-211,86){\tiny$T_-$}
\put(-211,51){\tiny$S_+$}
\put(-213,33){\tiny$D(\nu)$}
\put(-211,15){\tiny$S_-$}

\caption{Diagrammatic composition of two elements of $V$ as abstract strand diagrams. }\label{fig: comp in V}\end{figure}
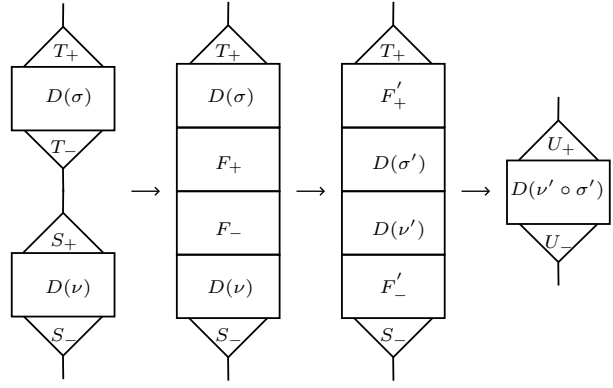

\subsection{ Links from $F$ and $T$}\label{sec: FT links}

We now discuss existing constructions of links from Thompson's groups $F$ and $T$. Jones initialized this program by providing a surjective map from $F$ to the set of isotopy classes of links in $S^3$ \cite[Theorem 5.3.1]{jones_unitary}. Given an element of $g \in F$, one uses its reduced pair of trees to create a link diagram $\mathscr{L}(g)$ by first making the trees ternary, then connecting the top vertex to the bottom vertex, and finally turning all $4$-valent vertices into crossings, as in Figure \ref{fig: links from F}. 
\begin{figure}\includegraphics[scale=0.4]{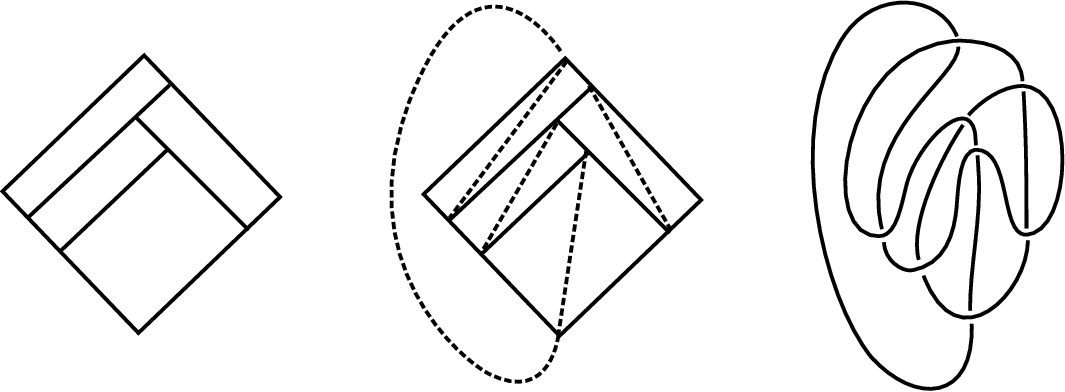}
\caption{Jones' method of associating links in $S^3$ to elements of $F$.}\label{fig: links from F}\end{figure}Note that a \textit{reduced} pair of trees is necessary; if one uses an unreduced pair, each canceling caret introduces an extra unlinked and unknotted component.

Jones also introduced the oriented subgroup $\vec{F}\leq F$ from which all link diagrams, when given the checkerboard shading, bound orientable surfaces. If one fixes the convention that the leftmost face of the resulting checkerboard surface is positively oriented, then all elements of $\vec{F}$ produce links with a natural choice of orientation. Algebraic properties of this new group were studied in \cite{golansapir}. Jones also showed how to associate links in $S^3$ to elements of $T$, however, $T$ is ``too big" in the sense that its subgroup $F$ already gives rise to all link types. However, Jones' map from $T$ gave rise to a new oriented subgroup $\vec{T}$, whose algebraic properties were further studied in \cite{nikkelren}.

Work of the second author used cylindrical strand diagrams to produce links in the thickened annulus from elements of $T$ (see Figure \ref{fig: links from T}), recovering the usual links of Jones from $F$ in a contractible subset of the thickened annulus. This map realizes all checkerboard colorable annular links \cite[Theorem 1.1]{LL25}. As was the case for $F$, working with unreduced triples introduces an extra unlinked unknot for each canceling caret. Annular links arising from the subgroup $\vec{T}$ inherit a natural orientation induced by the orientation of the checkerboard surface.

\begin{figure}\includegraphics[scale=0.4]{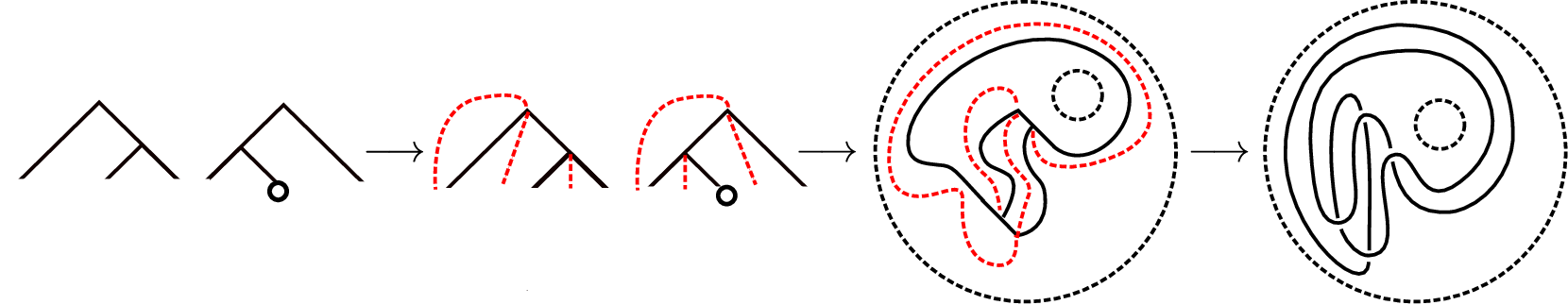}\caption{A link in the thickened annulus associated to $h \in T$.}\label{fig: links from T}\end{figure}

We note that all links in $S^3$ are checkerboard colorable; therefore both $F$ and $T$ gives rise to all CC links in $S^3$ and the thickened annulus, respectively. Moreover, both realization results rely on specific choices of checkerboard surfaces. Theorems \ref{thm_intro_real} and \ref{thm_intro_oriented} continue this pattern by constructing maps from $V$ (resp. $\vec{V}$) giving rise to all CC virtual (resp. almost classical) links. Toward this result, we now introduce virtual links, CC virtual links, and AC links.

\subsection{Virtual links} \label{sec_virtual_links}  A \emph{virtual link diagram} of $n$-components is an $n$-component link diagram in $\mathbb{R}^2$ where the crossings are marked as either classical over/under crossings or as virtual crossings (see Figure \ref{fig_crossings}). If each of the components carries a direction of travel, the diagram is said to be \emph{oriented}. Two virtual link diagrams are said to be \emph{equivalent} if they may be obtained from one another by a finite sequence of classical Reidemeister moves and detour moves (see Figure \ref{fig_moves}). The detour move allows one to arbitrarily place a portion of an arc which contains only virtual crossings. An equivalence class of (oriented) virtual link diagrams will be called a \emph{(oriented) virtual link type}. 

\begin{figure}[htb]
\begin{tabular}{ccc}
\begin{tabular}{c}
\includegraphics[width=.5in]{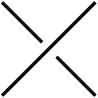} \\ \\ \underline{classical:} \end{tabular} \quad \quad
&\quad \begin{tabular}{c} \includegraphics[width=.5in]{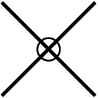} \\ \\ \underline{virtual:} \end{tabular} \quad &\quad \quad\begin{tabular}{c} \includegraphics[width=1in]{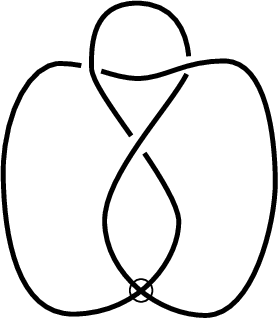} \end{tabular}
\end{tabular}
\caption{A virtual knot diagram (right) and its crossing types (left, center).}\label{fig_crossings}
\end{figure}

\begin{figure}[htb]
\begin{tabular}{ccc}
\begin{tabular}{c}
\includegraphics[width=1.1in]{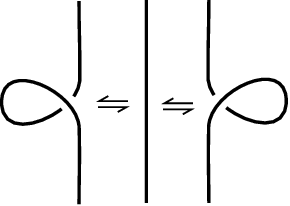} \\ \underline{$\Omega 1$:} \end{tabular} \quad
&\quad \begin{tabular}{c} \includegraphics[width=.7in]{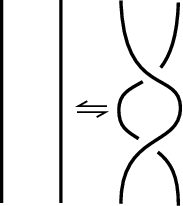} \\ \underline{$\Omega 2$:} \end{tabular} \quad &\quad \begin{tabular}{c} \includegraphics[width=2in]{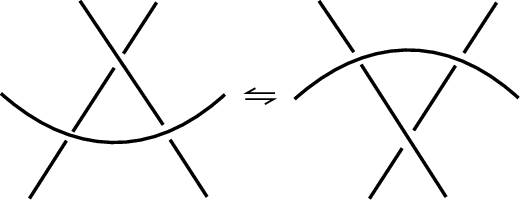} \\ \underline{$\Omega 3$:} \end{tabular}\\ \\ \multicolumn{3}{c}{\begin{tabular}{c}\includegraphics[width=2.2in]{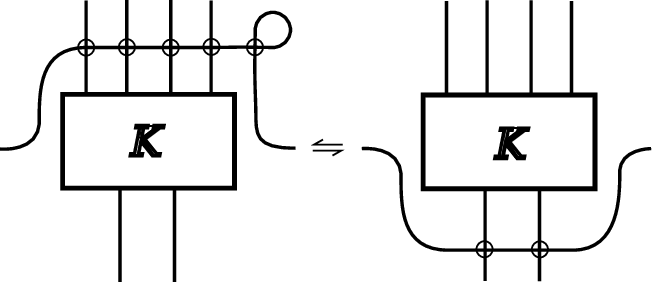}\\ \underline{A detour move:} \end{tabular}}
\end{tabular}
\caption{Reidemeister moves ($\Omega 1$, $\Omega 2$, $\Omega 3$) and detour moves.} \label{fig_moves}
\end{figure}

Carter, Kamada, and Saito \cite{CKS} showed that every virtual link diagram can be represented as a usual link diagram on closed oriented surface. This can be done by drawing a small disc (i.e. a $0$-handle) around each classical crossing, thickening each arc between the discs to an untwisted $1$-handle, and capping off each of boundary components with a disc (i.e. a 2-handle). Note that at a virtual crossing, the two bands pass over one another, so that the virtual crossing is removed. See Figure \ref{fig_hopf_on_torus}. The smallest genus among all closed orientable surfaces on which a representative of a virtual link type can be drawn is called its \emph{virtual 2-genus}.  A minimal genus representative of a virtual link type is unique up to diffeomorphism of the supporting surface (Kuperberg \cite{kuperberg}).   

\begin{figure}[htb]
\[
\xymatrix{
\begin{array}{c}
\includegraphics[width=1.5in]{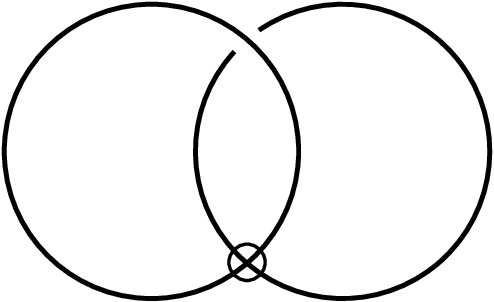}
\end{array} \ar[r] & \begin{array}{c}
\includegraphics[width=1.7in]{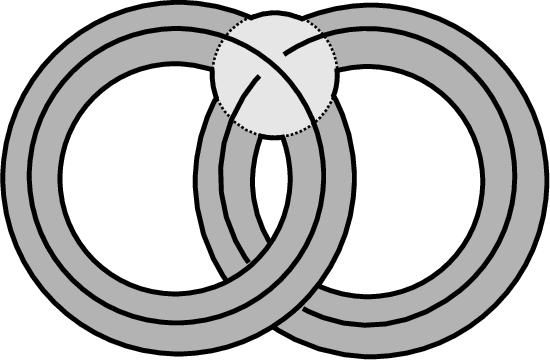}
\end{array} \ar[r] & \begin{array}{c}
\includegraphics[width=1.9in]{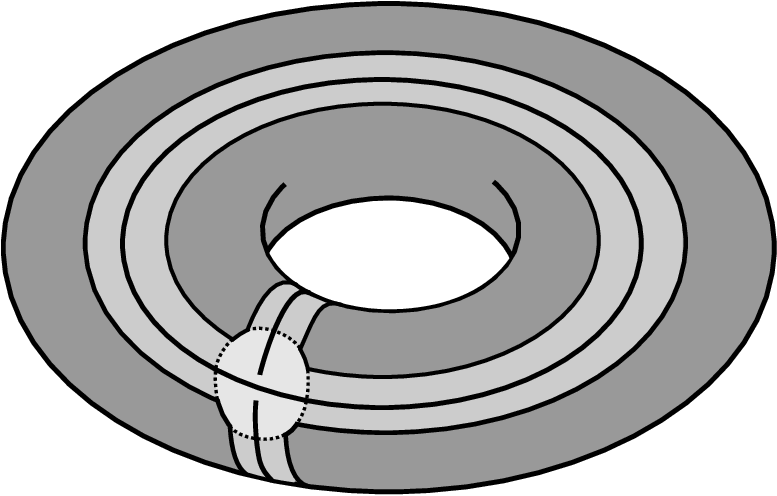}
\end{array}
}
\]
\caption{Realizing virtual link diagrams on a surface.} \label{fig_hopf_on_torus}
\end{figure}

A virtual link type $L$ is said to be \emph{checkerboard colorable (CC)} if $L$ can be represented as a diagram $D$ on a surface $\Sigma$ such that $[D]=0 \in H_1(\Sigma;\mathbb{Z}_2)$ (see Carter et al. \cite{CESSW}). Equivalently, a CC link on $\Sigma$ satisfies the property that the components of $\Sigma \smallsetminus D$ can be shaded in two colors, called \emph{black} and \emph{white}, so that every edge of $D$ bounds a component in each color (see N. Kamada \cite{n_kamada_cc}). An example is shown in Figure \ref{fig_check_colored}. If a virtual link type is checkerboard colorable, then it can be represented by a $\mathbb{Z}_2$-homologically trivial diagram on a minimal genus surface (see \cite{boden_chrisman_karimi}, Theorem 1.8).  If a checkerboard colorable virtual link type $L$ is furthermore represented by diagram $D$ on $\Sigma$ such that $[D]=0 \in H(\Sigma;\mathbb{Z})$, then $L$ is said to be \emph{almost classical (AC)}. In this case, a minimal genus representative is likewise null homologous. 

\begin{figure}[htb]
\begin{tabular}{ccc}
\includegraphics[width=1.5in]{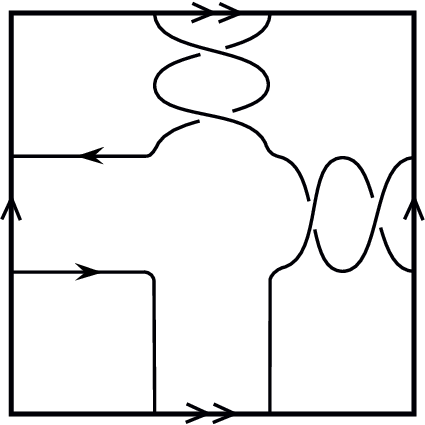}
& \includegraphics[width=1.53in]{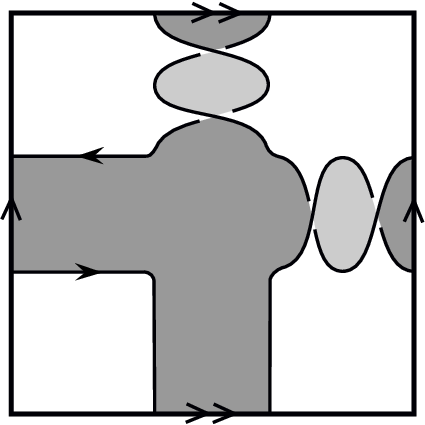} & \includegraphics[width=1.5in]{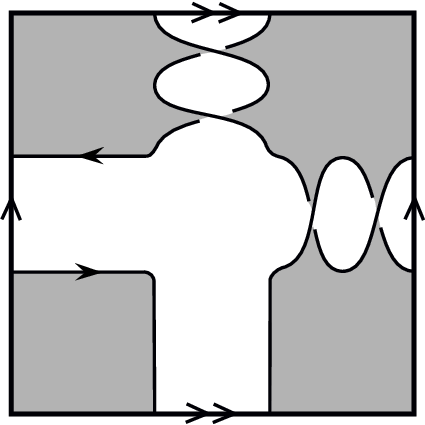}
\end{tabular}
\caption{Two checkerboard surfaces for a knot diagram on $S^1 \times S^1$.}
\label{fig_check_colored}
\end{figure}

For a $\mathbb{Z}_2$-null homologous link $\mathcal{L}$ in a thickened surface $\Sigma \times [0,1]$, a \emph{spanning surface} of $\mathcal{L}$ is compact surface $M \subset \Sigma \times [0,1]$, having no closed components, such that $\partial M=\mathcal{L}$. By adding compressible tubes in $\Sigma \times [0,1]$ between the components of a spanning surface if necessary, it can always be assumed that a spanning surface in $\Sigma \times [0,1]$ is connected if $\Sigma$ is connected. Of course, if $\Sigma$ itself is not connected, then its components can be joined by connected sums to obtain a connected surface with the same genus. If a CC link diagram $D$ on $\Sigma$ is a cellular embedding, so that $\Sigma \smallsetminus D$ is a union of discs, then both the black and white checkerboard surfaces for $D$ are spanning surfaces for the link $\mathcal{L} \subset \Sigma \times [0,1]$ corresponding to $D$. A minimal genus representative of a virtual link necessarily has a cellularly embedded diagram. Furthermore, if a virtual link is almost classical, then there is a minimal genus representative that has a checkerboard surface which is orientable (see \cite{boden_chrisman_karimi}, Proposition 5.8). In Figure \ref{fig_check_colored}, the given knot is $\mathbb{Z}$-null homologous and hence represents an almost classical knot. An oriented checkerboard surface is shown in black in the center of Figure \ref{fig_check_colored}. The black checkerboard surface on the right-hand side is non-orientable. Note that, in contrast to the classical case, it is not true that any two spanning surfaces of a link in $\Sigma \times [0,1]$ are $S^*$-equivalent. In fact, if $\Sigma$ has genus at least 1, there are exactly two $S^*$-equivalence classes of spanning surfaces for a given non-split link (see \cite{boden_chrisman_karimi}, Proposition 1.6). 

\subsection{Virtual spanning surfaces and Tait graphs} \label{sec_virt_tait_graphs} We now recall a method for drawing planar representations of spanning surfaces of knots in thickened surfaces, called \emph{virtual spanning surfaces}. These were first introduced in the oriented case in \cite{chrisman_vss,boden_chrisman_gaudreau} and in the unoriented case in \cite{boden_chrisman_karimi}. First, a \emph{virtual disc-band surface} is a finite collection of disjoint discs in the plane $D_1 \sqcup D_2 \sqcup \cdots \sqcup D_n$, which are connected by a finite collection of bands $B_1,\ldots,B_k$, with each band meeting a disc only along its attaching region. The bands may have classical half-twists. Two bands may intersect one another in a classical band crossing or a virtual band crossing (see Figure \ref{fig_band_crossings}). In particular, note that a virtual twist in a band is not allowed. Virtual crossings only appear in band crossings. An example of a virtual disc-band surface is given in Figure \ref{fig_virtual_disc_band_example}. 

\begin{figure}[htb]
\begin{tabular}{ccccccc}
\begin{tabular}{c}
\includegraphics[width=1in]{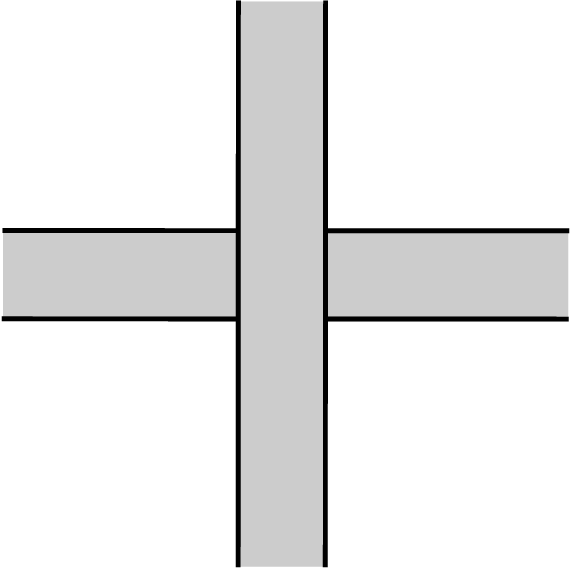} \end{tabular}
& & \begin{tabular}{c}\includegraphics[width=1in]{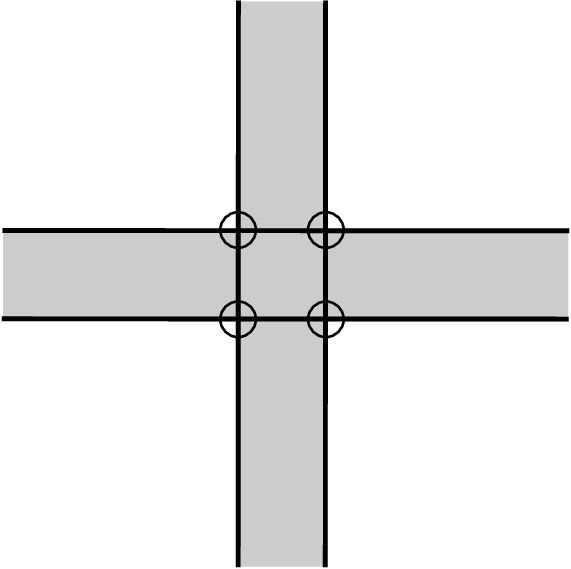}\end{tabular} & & \begin{tabular}{c} \\ \includegraphics[width=.6in]{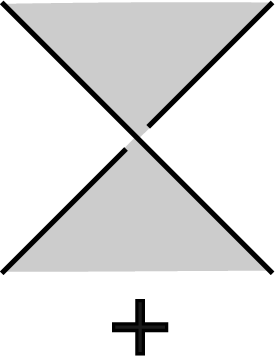} \end{tabular} & & \begin{tabular}{c} \\ \includegraphics[width=.6in]{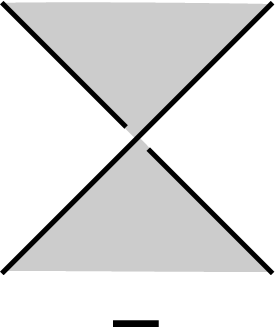} \end{tabular}\\
\underline{Classical band crossing:} & & \underline{Virtual band crossing:} & & \underline{Positive half-twist:} & & \underline{Negative half-twist}
\end{tabular}
\caption{Band crossings and half-twists in a virtual disc-band surface.} \label{fig_band_crossings}
\end{figure}

\begin{figure}
    \centering
    \includegraphics[width=2in]{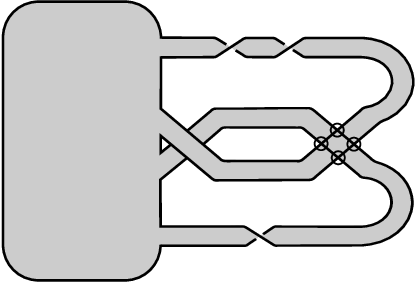}
    \caption{A virtual disc-band surface with a disc, two bands, a virtual band crossing, a classical band crossing, two positive half-twists and a negative half-twist.}
    \label{fig_virtual_disc_band_example}
\end{figure}

An arbitrary spanning surface $M \subset \Sigma \times [0,1]$ of a link $\mathcal{L} \subset \Sigma \times [0,1]$ can be drawn as a virtual disc-band surface as follows. There is an isotopy taking $M$ to a disc-band surface in $\Sigma \times [0,1]$ (see \cite{chrisman_vss}, Theorem 4), i.e. a decomposition of $M$ into $0$-handles and $1$ handles embedded in $\Sigma \times [0,1]$.  We will assume that $M$ is already in this form. It may also be assumed that $M$ has a projection to $\Sigma$ so that the discs project to disjoint discs on $\Sigma$ and the image of the bands contain only classical half-twists and band crossings. Now, remove a small disc $D$ centered at a point $\infty \in \Sigma$ that is disjoint from the projection of $M$ to $\Sigma$. Assuming $\Sigma$ is a closed connected surface of genus $g$, then $\Sigma_{\infty}=\Sigma \smallsetminus \text{int}(D)$ itself has a disc-band decomposition consisting of a single $0$-handle $H^0$ and $1$-handles $H^1_1,\ldots,H_{2g}^1$. Place $H^0$ in the plane and attach the $1$-handles as shown in Figure \ref{fig_screen_definition}, right. The overlap of the bands $H_{2i-1}^1$ and $H_{2i}^1$ is called a \emph{virtual region}, and is denoted with four virtual crossings as shown in Figure \ref{fig_screen_definition}. The entire figure $S$ in $\mathbb{R}^2$ is called the \emph{screen}.  The disc-band surface $M$ can now be transferred to $S$. After an isotopy, it may be assumed that both the discs of $M$ and all of its band crossings occur inside $H^0$. Hence, the $1$-handles $H^1_j$ contain only bands of $M$. If two bands of $M$ overlap in a virtual region, the intersections are drawn as a virtual crossing of bands. The image of $M$ on the screen is then a virtual disc-band surface. The image of its boundary $\partial M$ in $\mathbb{R}^2$ is a diagram of the virtual link type of $\mathcal{L}$. For an example, see Figure \ref{fig_disc_band_example}.

\begin{figure}[htb]
\[
\xymatrix{
\begin{array}{c}
\includegraphics[width=2.5in]{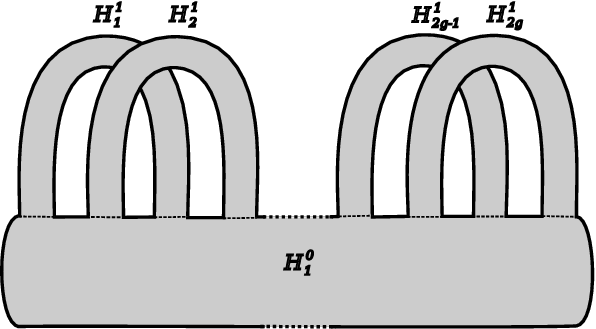}
\end{array} \ar[r] & \begin{array}{c}
\includegraphics[width=2.5in]{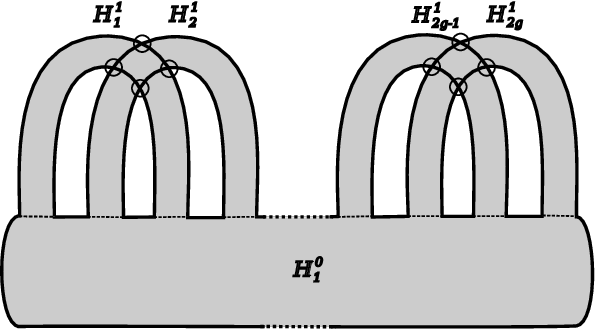}
\end{array}
}
\]
\caption{A genus $g$ surface $\Sigma_{\infty}$ with one boundary component (left) and a screen $S$ in $\mathbb{R}^2$ to which link diagrams on $\Sigma_{\infty}$ can be projected to virtual links (right).} \label{fig_screen_definition}
\end{figure}

\begin{figure}[htb]
\[
\xymatrix{
\begin{array}{c}
\includegraphics[width=1.75in]{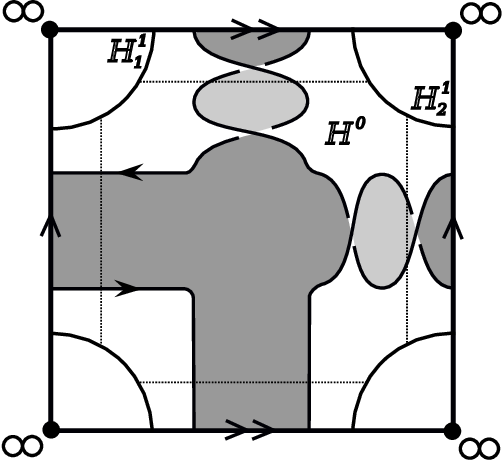}
\end{array} \ar[r] & \begin{array}{c}
\includegraphics[width=2.5in]{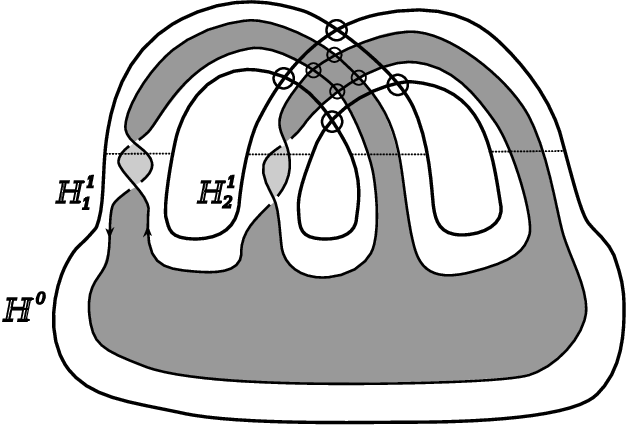}
\end{array}
}
\]
\caption{A disc-band surface and a corresponding virtual disc-band surface.} \label{fig_disc_band_example}
\end{figure}

Finally, we recall from \cite{boden_chrisman_karimi} how Tait graphs are constructed for checkerboard colorable virtual links. Suppose that $D$ is a cellularly embedded, checkerboard colorable link diagram on $\Sigma$. Fix a checkerboard coloring $\xi$ and let $M_{\xi}$ denote one of the two checkerboard surfaces. For concreteness, we will say it is the black one. Contract $M_{\xi}$ to its $1$-skeleton, so that each of the black cells of $M_{\xi}$ is contracted to a point and each of the twisted bands at the crossings is contracted to its core. Then the \emph{Tait  graph} $\Gamma_{\xi}$ is the graph on $\Sigma$ whose vertices are in one-to-one correspondence with the black cells of $\xi$ and whose edges correspond to the crossings of $D$. The edges are signed $\pm$ according to the convention shown on the right hand side of Figure \ref{fig_band_crossings}. If, in addition, $M_{\xi}$ is oriented, then each vertex of $\Gamma_{\xi}$ is also signed $+$ or $-$ according to whether the orientation of the corresponding disc agrees or disagrees, respectively, with the orientation of $\Sigma$. This notation will be used ahead in Section \ref{sec: oriented} when we discuss the oriented subgroup $\vec{V}$ of $V$.

By a \emph{virtual Tait graph}, we mean a Tait graph for a checkerboard surface $M_{\xi}$ on a surface $\Sigma$ that has been projected onto a fixed screen $S \subset \mathbb{R}^2$ of  the punctured background surface $\Sigma_{\infty}$.  Note that in a virtual Tait graph, it may be assumed that all the vertices occur in the $0$-handle $H^0$ of $S$. Although edges of $\Gamma_{\xi}$ intersect in $\Sigma$ only at vertices, edges in $S$ may also intersect in the virtual regions, so that transversal intersections of edges are marked as virtual crossings. Also, note that the image of $\Gamma_{\xi}$ in $S$ retains a cyclic ordering of the edges around a vertex, and this will be considered as part of the given data of a virtual Tait graph. For the spanning surface shown in Figure \ref{fig_disc_band_example}, the Tait graph and virtual Tait graph are shown in Figure \ref{fig_virt_tait_graph_example}.

\begin{figure}[htb]
\[
\xymatrix{
\begin{array}{c}
\includegraphics[width=1.25in]{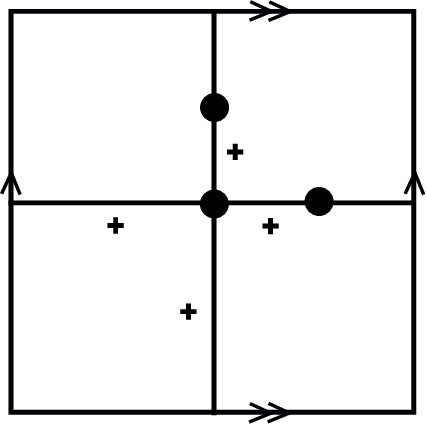}
\end{array} \ar[r] & \begin{array}{c}
\includegraphics[width=2in]{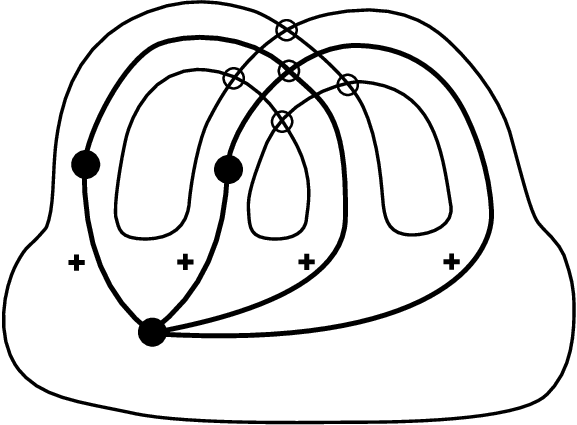}
\end{array} \ar[r] & \begin{array}{c}
\includegraphics[width=1.5in]{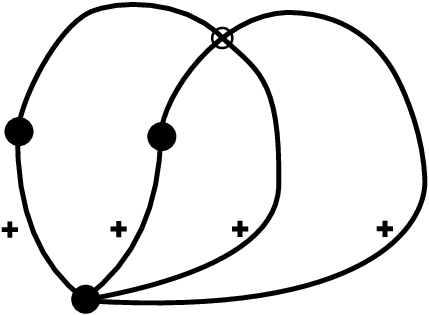}
\end{array}
}
\]
\caption{Tait graph and a virtual Tait graph for the spanning surface in Figure \ref{fig_disc_band_example}.} \label{fig_virt_tait_graph_example}
\end{figure}

\section{Construction and Realization of CC Virtual Links from $V$}\label{sec: construction}
This section contains the proof of Theorem \ref{thm_intro_real}. In Section \ref{sec_thm_intro_real_forward}, we show how to obtain a checkerboard colorable virtual link diagram for every element of $V$. For the reader's convenience, we review the proof of Jones' realization theorem for $F$ in Section \ref{sec_jones_review}. In Section \ref{sec_real}, this is generalized to $V$. An example of the realization algorithm is given in Section \ref{sec_example}. A comparison with the Kodama-Takano realization algorithm  \cite{KTVF} for $\mathit{VF}$ is given in Section \ref{sec_vf}. 
\subsection{From Thompson's Group $V$ to CC Virtual Links} \label{sec_thm_intro_real_forward}
Let $(T_+,T_-,\sigma)$ be a reduced triple representing $v \in V$. To construct a checkerboard colorable virtual link $\mathscr{L}_V(v)$, we first work with each tree separately. From each tree, we build a tangle using the same process as in \cite{jones_unitary, jones18}, which was described in Section \ref{sec: FT links}.  We checkerboard-shade the resulting tangle diagram, keeping the outside face unshaded. The resulting shaded surface has, for each leaf of the tree, one band meeting the boundary of the unit cube transversely. We now attach the bands of the two surfaces according to the permutation $\sigma$, creating virtual crossings of bands whenever they overlap. The virtual link $\mathscr{L}_V(v)$ is the boundary of the resulting checkerboard surface, see Figure \ref{fig: cc links from V}. 

\begin{figure}
\includegraphics[scale=0.4]{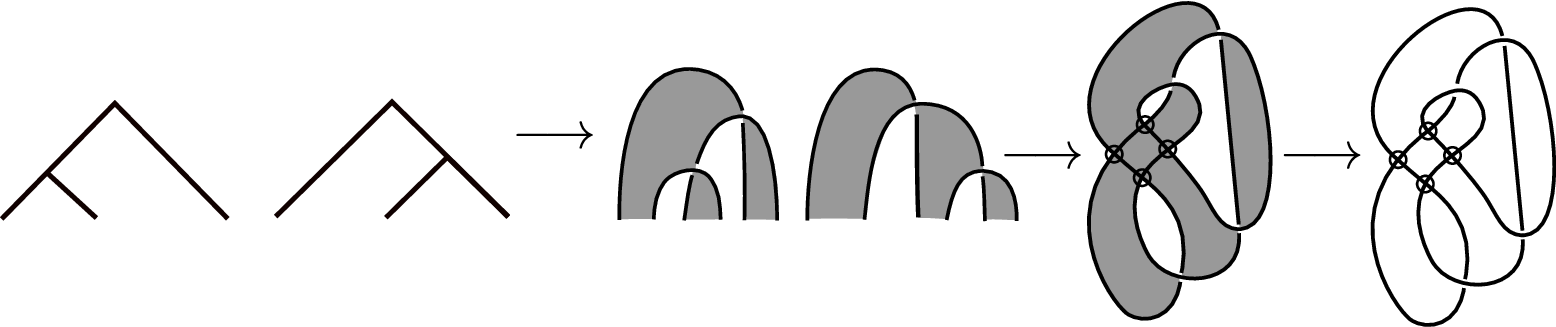}
\put(-250,10){$2$}
\put(-229,10){$1$}
\put(-205,10){$3$}
\caption{Constructing checkerboard-colorable virtual links from elements of $V$.}\label{fig: cc links from V}\end{figure}

\begin{lemma} For all $v \in V,\mathscr{L}_{V}(v)$ is checkerboard colorable. \end{lemma}

Note that for $v \in F$, the bands are attached from left to right, recovering Jones' map from $F$ to links in $S^3$. The annular construction $\mathscr{L}_{\mathbb{A}}$ of \cite{LL25} is also recovered in the following sense:

\begin{lemma}\label{lem: TV diagrams} 
    For $v \in T$, $\mathscr{L}_{V}(v)$ can be realized as link in the thickened annulus that is virtually equivalent to $\mathscr{L}_{\mathbb{A}}(v)$, via a diffeomorphism of the annulus.
\end{lemma} 
\begin{proof}
Let $(T_+,T_-,\sigma)$ be a triple for $v \in V$, and let $n$ refer to the number of leaves in $T_+$ and $T_-$. We will use the cyclic nature of $\sigma$ to realize $\mathscr{L}_{V}(v)$ in an annulus with two half-twists, pictured in Figure \ref{fig: TV equiv}. We divide the annulus into three parts--a top thickened semicircle in which we will embed the tangle corresponding to $T_+$, a bottom thickened semicircle in which we will embed the tangle corresponding to $T_-$, and a middle portion consisting of two bands, one crossing over the other, in which we will embed the strands connecting the two tangles.

Focusing on the middle portion of the twisted annulus, let us refer to the band that passes over from top right to bottom left as the \textit{top band} and the band below it as the \textit{bottom band.} If $\sigma(i)< i$, then the two strands corresponding to the $i$th leaf of $T_+$ will run parallel along the top band. Otherwise, the two strands corresponding to the $i$th leaf of $T_+$ will run along the bottom band. Because $\sigma$ is cyclic, we may apply this rule for all $1 \leq i \leq n$ and the $2n$ corresponding strands will remain embedded. A diffeomorphism untwisting the annulus returns $\mathscr{L}_{\mathbb{A}}(v)$.
\end{proof}

\begin{figure}\includegraphics[scale=0.5]{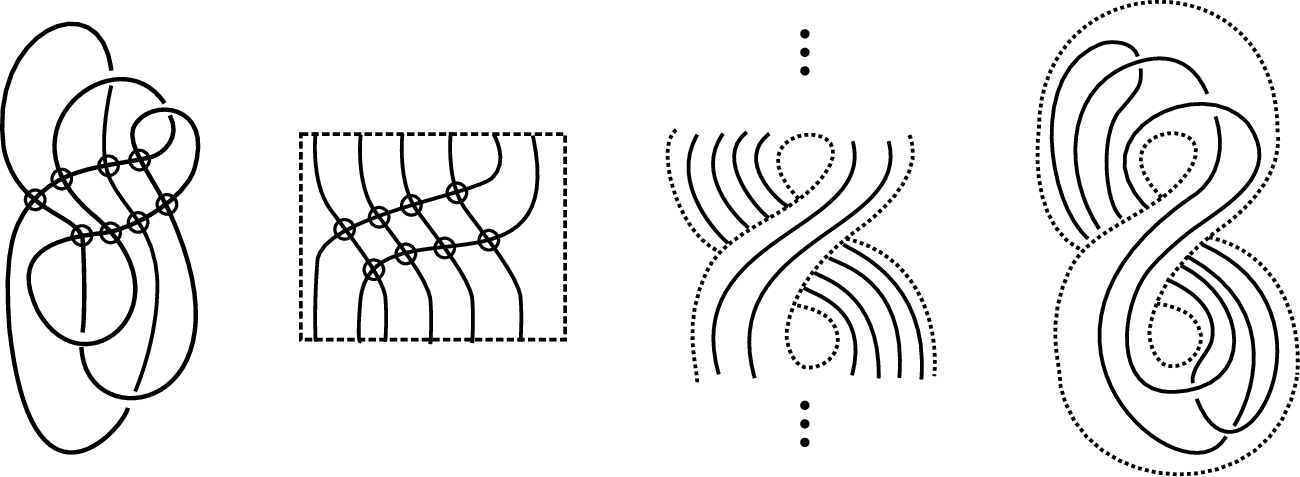}\caption{The virtual link for $h \in T$ realized on the thickened annulus; compare with Figure \ref{fig: links from T}.}\label{fig: TV equiv} \end{figure}

We note that because more relations are available in virtual knot theory than annular knot theory, $\mathscr{L}_{\mathbb{A}}(v)$ and $\mathscr{L}_{V}(v)$ may be virtually equivalent but not isotopic in the thickened annulus.

\begin{lemma}\label{lem: cancelingcarets-resultsinunknot} Let $(S_+,S_-,\tau)$ be the result of adding a canceling caret to some reduced triple $(T_+,T_-,\sigma)$. Then $\mathscr{L}_{V}(S_+,S_-,\tau)$ differs from $\mathscr{L}_{V}(T_+,T_-,\sigma)$ by a disjoint copy of the unknot. \end{lemma}

\begin{proof}
        The canceling caret introduces a unknotted component $U$. Then $U$ has two classical crossings and maybe some virtual crossings with the rest of the link diagram (see, for example Figure \ref{fig: TV carets}). The classical crossings are over-crossings by the unknotted component. Those arcs on the link diagram having virtual crossings with $U$ can be detoured over the over-crossings (i.e. either the local maximum or minimum). Afterwards, $U$ has only two over-crossings and no virtual crossings. Using a Reidemeister 2 move, $U$ can made disjoint for the rest of the diagram.
\end{proof}

\begin{figure}\includegraphics[scale=0.4]{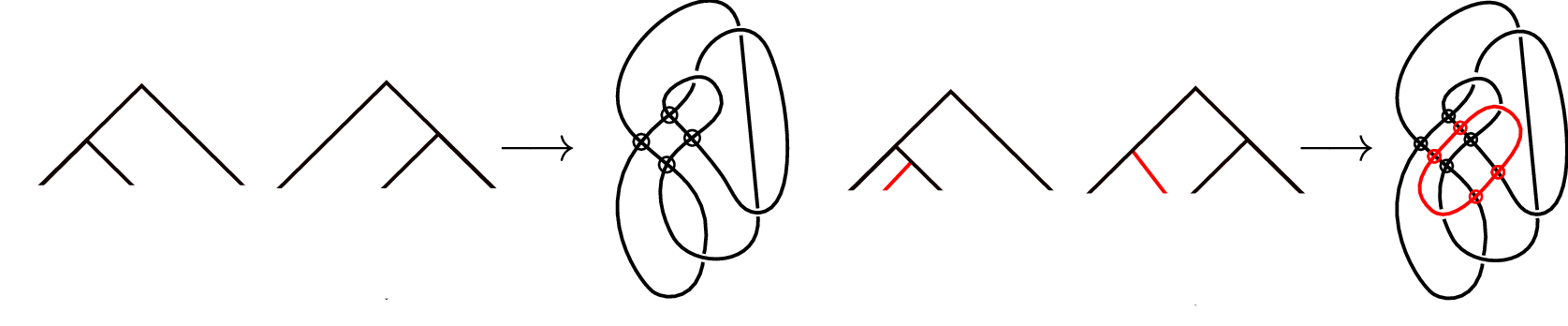}
\put(-267,17){$2$}
\put(-223,17){$1$}
\put(-246,17){$3$}
\put(-102,17){$2$}
\put(-154,20){,}
\put(-88,17){$3$}
\put(-77,17){$1$}
\put(-57,17){$4$}
\caption{A canceling caret in $V$ producing an extra unlinked unknotted component of the link. }\label{fig: TV carets}\end{figure}
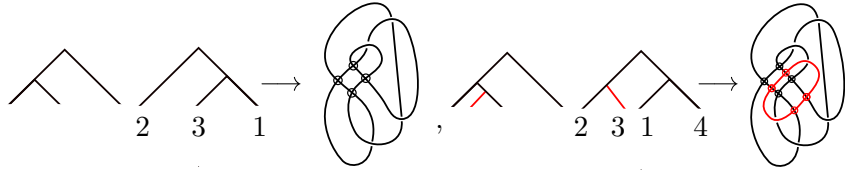

\subsection{Review: Jones' realization theorem for $F$}\label{sec_jones_review}
 Before continuing with the proof of Theorem \ref{thm_intro_real}, it is instructive to first review Jones' realization theorem for $F$ (see \cite{jones_unitary}, Sections 4.1 and 5.3.1). Define $\mathbb{R}^2_+=\{(x,y) \in \mathbb{R}^2|y \ge 0\}$ and $\mathbb{R}^2_-=\{(x,y)\in\mathbb{R}^2|y \le 0\}$. The main idea is to relate pairs $(T_+,T_-)$ of planar rooted bifurcating trees to Tait graphs of links. Let $N$ be the number of leaves of $T_{\varepsilon}$, where $\varepsilon=\pm$. First we embed $T_{\varepsilon}$ into $\mathbb{R}^2_{\varepsilon}$ so that its leaves lie on the coordinates $(1/2,0),(3/2,0),\ldots,(2N-1)/2$ and so that each edge is a line segment having slope equal to either $+1$ or $-1$. Below (resp. above) each non-leaf vertex of $T_+$ (resp. $T_-$) we also require that the left edge has positive (resp. negative) slope and the right edge has negative (resp. positive) slope. The next step is to construct a second pair $G(T_+,T_-)=(G_+,G_-)$ of trees as follows. The vertices of $G_{\varepsilon}$ are at coordinates $(0,0),(1,0),\ldots,(N,0)$. Note that in each region in $\mathbb{R}^2_{\varepsilon}$ defined by $T_{\varepsilon}$, there is exactly one vertex of $G_{\varepsilon}$. For every edge $e$ of $T_{\varepsilon}$ having slope of sign $\varepsilon$, we draw an edge of $G_{\varepsilon}$ between the vertices in the regions bordered by $e$. We will typically draw these edges as semi-elliptical arcs transversal to the edge of $T_{\varepsilon}$. See Figure \ref{fig_t_to_g}. To make $G(T_+,T_-)$ into a Tait graph, we give all the edges in $G_{\varepsilon}$ the sign $\varepsilon$. Then when $(T_+,T_-)$ is a reduced pair, and $\mathscr{L}(T_+,T_-)$ is given its standard checkerboard coloring $\xi$, the Tait graph of $\xi$ is $G(T_+,T_-)$.

\begin{figure}[htb]
\[
\begin{array}{ccc}
\begin{array}{c}
\includegraphics[width=2in]{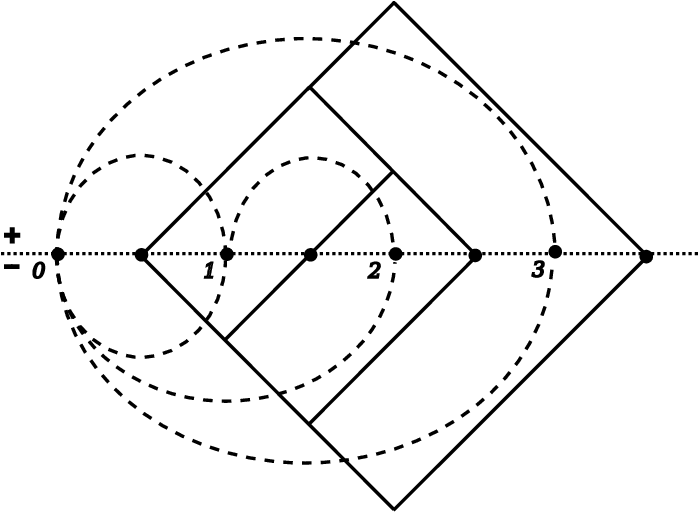}
\end{array} & \begin{array}{c}
\includegraphics[width=1.8in]{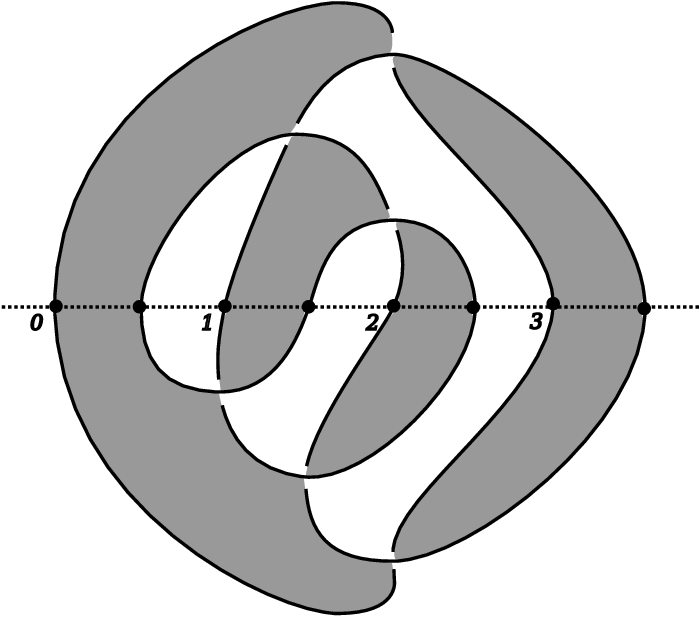}
\end{array} & \begin{array}{c}
\includegraphics[width=2.2in]{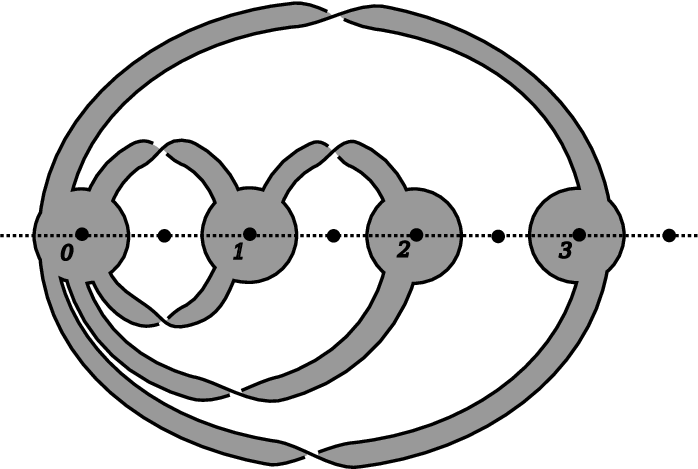}
\end{array}
\end{array}
\]
\caption{ (Left) Constructing $T_{\varepsilon}$ (solid lines) from $G_{\varepsilon}$ (dashed semicircles) and vice versa. The center shows the checkerboard surface for $(T_+,T_-)$ and the right-hand sides shows the (isotopic) checkerboard surface for the Tait graph $G(T_+,T_-)$.}\label{fig_t_to_g}
\end{figure}

Ideally, one would begin with a checkerboard coloring of a link diagram $D \subset \mathbb{R}^2$ with Tait graph $G$ and assemble a pair of planar rooted bifurcating trees $(T_+,T_-)$ so that $G(T_+,T_-)=G$. It would then follow that every diagram $D$ can be realized as $\mathscr{L}(T_+,T_-)$ for some reduced pair $(T_+,T_-)$. Unfortunately, it may not be possible to do this with the original Tait graph $G$. Jones' fix is to modify $G$ with a set of moves that preserve the link type. These moves, called Type RI, Type RIIa, and Type RIIb are shown in Figure \ref{fig_type_moves}. They correspond in the evident way to Reidemeister 1 and 2 moves on the link diagram $D$. Graphs related by these moves are said to be \emph{$2$-equivalent}.

Jones then identifies a condition on a graph $G$, called \emph{Thompson} (see \cite{jones_unitary}, Definition 5.3.9), and gives an algorithm for drawing a uniquely defined pair of planar rooted bifurcating trees $(T_+,T_-)$ such that $G=G(T_+,T_-)$ when $G$ satisfies the Thompson condition (\cite{jones_unitary}, Lemma 4.1.4). Lastly, the proof is completed by showing that every Tait graph is $2$-equivalent to a Thompson one (\cite{jones_unitary}, Lemma 5.3.13). This recipe is adapted to $V$ below but modified for non-planarity. 

\begin{figure}[htb]
\[
\begin{array}{c}
\xymatrix{\begin{array}{c}\includegraphics[width=.6in]{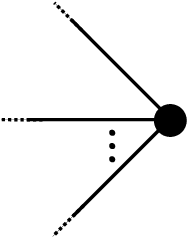} \end{array} \ar[r] & \ar[l] \begin{array}{c}\includegraphics[width=1in]{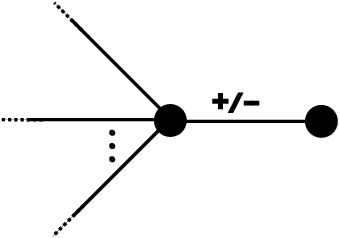} \end{array} } \\ \text{\underline{Type RI:}}
\end{array} \quad \begin{array}{c}
\xymatrix{\begin{array}{c}\includegraphics[width=2in]{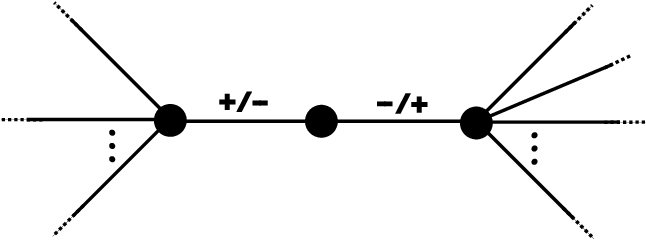} \end{array} \ar[r] & \ar[l] \begin{array}{c}\includegraphics[width=1in]{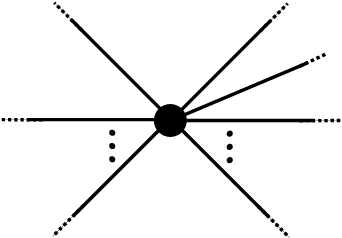} \end{array} } \\ \text{\underline{Type RIIa:}}
\end{array}
\]
\[
\begin{array}{c}
\xymatrix{\begin{array}{c}\includegraphics[width=1.5in]{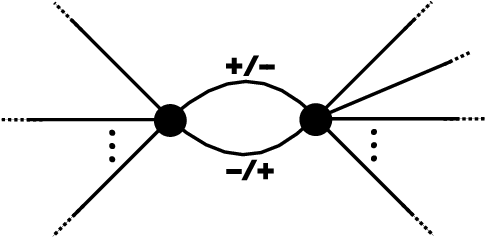} \end{array} \ar[r] & \ar[l] \begin{array}{c}\includegraphics[width=1.5in]{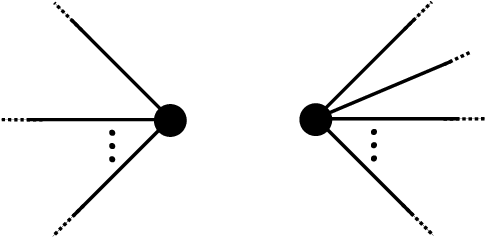} \end{array} } \\ \text{\underline{Type RIIb:}}
\end{array} 
\]
\caption{Moves generating 2-equivalence of Tait graphs.}\label{fig_type_moves}
\end{figure}

\subsection{Realizing CC virtual links} \label{sec_real} Now with Jones' argument for $F$ in mind, we proceed to the case of realizing all CC virtual links with $V$. Our first task is to extend the Thompson property of classical Tait graphs to virtual Tait graphs. Following \cite[Definition 5.3.5]{jones_unitary}, the initial step is to define a notion of \emph{standard} for virtual Tait graphs.

\begin{definition}[Standard] \label{def_standard} Let $G$ be a virtual Tait graph in $\mathbb{R}^2$ with $N+1$ vertices, $N>0$. Then $G$ will be called \emph{standard} if the following conditions are satisfied:
\begin{enumerate}
    \item its vertices are at the coordinates $(0,0), (1,0),\ldots, (N,0)$, 
    \item each edge $e$ can be parameterized by $(x_e(t),y_e(t))$, $0 \le t \le 1$, with $x_e'(t)>0$ for $0 <t<1$ and either $y_e(t)>0$ for $0<t<1$ or $y_e(t)<0$ for $0<t<1$, and
    \item all virtual crossings of $G$ lie in the first quadrant.
\end{enumerate}
A standard graph $G$ can be decomposed into two graphs $G_{\pm}=G \cap \mathbb{R}^2_{\pm}$, called the \emph{upper} and \emph{lower} subgraphs of $G$, respectively. Note that by definition $G_-$ has no virtual crossings. Note also that property (2) can be satisfied with a semi-elliptical arc with horizontal axis along the $x$-axis.
\end{definition} 

We will say that two graphs are \emph{virtually $2$-equivalent} if they can be related by a finite sequence of the moves in Figure \ref{fig_type_moves}, detour moves along edges which preserve the cyclic ordering along the vertices, and planar isotopies. Given two virtual Tait graphs $G_1,G_2$, let $L_1,L_2$ denote the virtual link diagrams reconstructed from them. Then if $G_1,G_2$ are virtually $2$-equivalent, then $L_1$ and $L_2$ are clearly equivalent virtual links. Hence, the next step in continuing Jones' argument is to show that every virtual Tait graph can be made standard by virtual $2$-equivalence.

\begin{figure}[htb]
\[
\xymatrix{
\begin{array}{c}\includegraphics[width=1.5in]{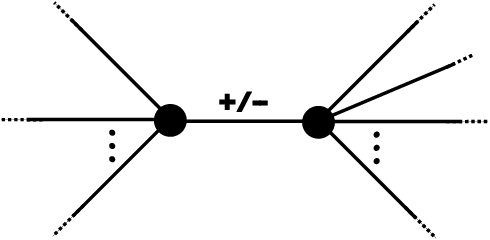} \end{array} \ar[r] & \ar[l] \begin{array}{c}\includegraphics[width=2.3in]{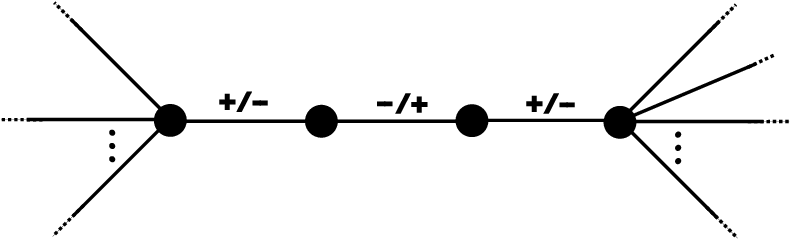} \end{array}} 
\]
\caption{Subdividing an edge with a Type RIIa move.}\label{fig_subdivide_edge}
\end{figure} 

\begin{lemma} \label{lemma_prepare_standard} Every CC virtual link type has a virtual Tait graph that is virtually 2-equivalent to a standard one. 
\end{lemma}

\begin{proof} Let $L$ be a CC virtual link type. Let $D$ be a diagram for $L$ on a surface $\Sigma$ of minimal genus. By the discussion in Sections \ref{sec_virtual_links} and \ref{sec_virt_tait_graphs}, it may be assumed that $\Sigma$ is connected and that $D$ is cellularly embedded. Choose a checkerboard coloring $\xi$ of $D$ and let $\Gamma_{\xi}$ be the Tait graph on $\Sigma$ for one of the checkerboard surfaces, say the black one. Let $\infty$ be a point in one of the white cells of $\xi$. From a handle decomposition $H^0 \cup H^1_1 \cup \cdots \cup H^{1}_{2g}$ of $\Sigma_{\infty}$, we draw the screen $S$ in $\mathbb{R}^2$ as follows. The $0$-handle is drawn as a rectangle $R$ in $\mathbb{R}^2$ with each side parallel to either the $x$-axis or $y$-axis. We position $R$ so that it is bisected by the $x$-axis and contains $(0,0)$ as an interior point. The $1$-handles are then placed as in Figure \ref{fig_screen_definition} so that they lie in the first quadrant. Next, we transfer $\Gamma_{\xi}$ to $S$, so that all of the virtual crossings in the virtual Tait graph $G$ appear in the first quadrant.

We now make a few initial preparations. First, it may be assumed that no edge has both its endpoints on the same vertex, since any edge may be subdivided using a Type RIIa move (see Figure \ref{fig_subdivide_edge}). Similarly, we may assume that there are no multiple edges, i.e. a set of two or more edges all adjacent to the same pair of vertices. Subdividing each of the multiples reduces to total number of such edges, so that by repeating this process we arrive at a graph with no multiple edges. Using further subdivisions, we may also assume that no edges with a virtual crossing are incident to the same vertex. After making these initial adjustments, an isotopy can be used to position the vertices at the required integer coordinates. This can be done by pushing all the vertices along the 1-handles in the screen so that they lie in $R$. Then draw a simple closed curve $R$ that passes through all the vertices. Straightening out this string puts all of the vertices on a line, which can then be assumed to coincide with the $x$-axis. The position of the vertices can then be adjusted to lie on the integer coordinates. It may of course be necessary to stretch $R$ so that its interior includes enough integer coordinates for the vertices of $G$. Note that by construction, the virtual crossings in $G$ only occur in the first quadrant.

Possibly, there is now some edge $e$ which crosses the $x$-axis. In this case, the edge can be subdivided using Type RIIa moves (as in Figure \ref{fig_subdivide_edge}) and the new vertices can be positioned along the $x$-axis (as in Figure \ref{fig_eliminate_axis_crossing}).  This can be continued until there are no edges crossing the $x$-axis. 

\begin{figure}[htb]
\[
\xymatrix{
\begin{array}{c}\includegraphics[width=1.4in]{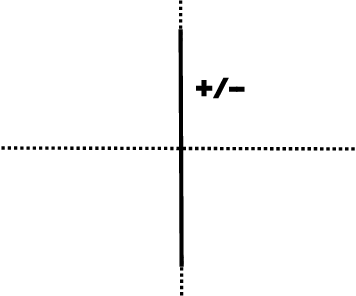} \end{array} \ar[r] &  \begin{array}{c}\includegraphics[width=1.3in]{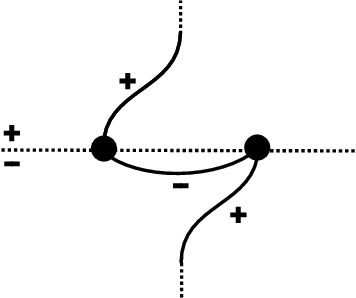} \end{array} & \hspace{-1cm}-\text{or}-\hspace{-1cm} & \begin{array}{c}\includegraphics[width=1.3in]{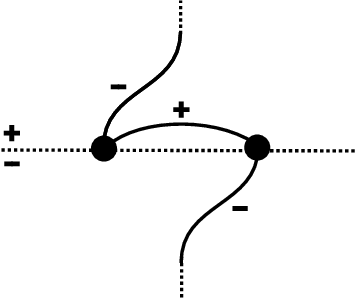} \end{array}} 
\]
\caption{Subdividing an edge which crosses the $x$-axis.} \label{fig_eliminate_axis_crossing}
\end{figure}

 All of the edges of $G$ now lie in either $\mathbb{R}^2_+$ or $\mathbb{R}^2_-$, all of the vertices are on integer coordinates $(0,0),\ldots,(N,0)$, and all of the virtual crossings are in $\mathbb{R}^2_+$. We may now dispense with the screen $S$, but in the end all virtual crossings will remain in $\mathbb{R}_+^2$. Since there are no edges having both endpoints on the same vertex, there is an isotopy taking each edge to a semicircle that lies in either $\mathbb{R}^2_+$ or $\mathbb{R}^2_-$.  It is important to observe that only the endpoints of an edge are important. The classical crossing of the corresponding band of the virtual disc-band surface can be pushed along the edge so that it lies close to one of its endpoints. Using a detour move, the band can then be redrawn as a semicircle edge. The initial preparation now guarantees that the semicircles are all distinct: two semicircles cannot coincide, since the graph has no loops or double edges. Furthermore, the isotopies and detour moves can be done while preserving the cyclic order of the edges around each vertex. This follows because of the assumption that no edges incident to the same vertex have virtual crossings. Hence, the final graph is standard and we are done.
\end{proof}

Now we can generalize Jones' notion of a Thompson graph to the non-planar setting.
\begin{definition}[Thompson graph] \label{def_thompson} A virtual Tait graph $G$, not necessarily planar, is said to be \emph{Thompson} if it is standard and its upper and lower subgraphs $G_{\pm}$ both have the property that every vertex other than $(0,0)$ is connected by an edge to exactly one vertex on its left, $G_{+}$ contains no negatively signed edges, and $G_{-}$ contains no positively signed edges. \end{definition}

As discussed in Section \ref{sec_jones_review}, Jones proved that every planar Thompson graph corresponds to a pair of planar bifurcating trees $(T_+,T_-)$ representing an element of $F$. In the case of $V$, it will be shown ahead that if $G$ is a Thompson virtual Tait graph, then there is a pair of planar bifurcating trees $(T_+,T_-)$ and a permutation $\sigma$ such that $(T_+,T_-,\sigma)$ represents an element of $V$. This leaves two tasks: (1) showing that every standard virtual Tait graph is virtually $2$-equivalent to a Thompson one and (2) identifying the planar trees $(T_+,T_-)$ and the permutation $\sigma$. The next lemma takes care of the first task and the second task is taken up immediately afterwards.
\begin{lemma} \label{lemma_virtual_thompson} A standard virtual Tait graph $G$ is virtually 2-equivalent to a Thompson one. 
\end{lemma}

\begin{proof} This follows exactly as in Jones \cite{jones_unitary}, Lemma 5.3.13. If one of the vertices of $G$ (other than $(0,0)$) has no arrows to the left, this deficiency can be removed by a Type RIIb move (\cite{jones_unitary}, top of page 37). If a vertex of $G$ has only one leftwards arrow, the missing leftwards edge of $G_+$ or $G_-$ can be added with a Type RI and a Type RIIb move (\cite{jones_unitary}, bottom of page 37). If one vertex of $G_{\pm}$ has more than one leftwards arrow, this can be fixed with a Type RIIa move, a Type RI move, and a Type RIIb move (see \cite{jones_unitary}, center of page 38). The last case is that $G_{\varepsilon}$ has an arrow signed $-\varepsilon$. In this case, there is a rather complicated sequence of Type RI, Type RIIa, and Type RIIb moves (see \cite{jones_unitary}, top of page 39) which ultimately changes the sign of the offending edge. In each case, the modifications used occur in a disc neighborhood of one or two vertices, and hence, can be assumed to occur within the $0$-handle $H^0$. As this has no effect on the virtual crossings, the same argument works if $G$ is non-planar.  \end{proof}

The last step in the argument is to express a Thompson virtual Tait graph $G=G_+ \cup G_-$ as a triple consisting of two planar rooted bifurcating trees $T_+,T_-$ and a permutation $\sigma$. Given such a decomposition, \cite{jones_unitary}, Lemma 4.1.4, states that $(T_+,T_-)$ corresponds to an element of $F$. Adding the permutation back in then yields an element of $V$. To show that the vertices in $G_+$ can be permuted to obtain a planar (but still Thompson!) graph, we will use the following notation. Suppose $G$ is Thompson with $N+1$ vertices and let $G_+$ be its upper subgraph. For $\sigma \in S_N$, let $\sigma \cdot G_+$ denote the graph obtained from $G_+$ by permuting the vertices $1,\ldots,N$ according to $\sigma$ and reconnecting the edges. The vertex at $0$ is stationary. Note that as long as the cyclic ordering of the edges at each vertex is preserved and the vertex at $0$ is stationary, the manner in which the edges are reconnected to their vertices is arbitrary since any two ways of doing so are related by detours and planar isotopies. It is helpful to  visualize this as stacking $G_+$ on top of the abstract strand diagram $\beta_{\sigma}$ for $\sigma$, as shown in Figure \ref{fig_stacking_permuation}. To make our conventions match those of the preceding sections at the end of the argument, here the permutation $\sigma$ must be oriented \emph{downwards}. This is emphasized in Figure \ref{fig_stacking_permuation}, where the arrows are oriented from the top of the box to the bottom.

\begin{figure}[htb]
\[
\xymatrix{
\begin{array}{c} \vspace{-1cm} \includegraphics[width=2.3in]{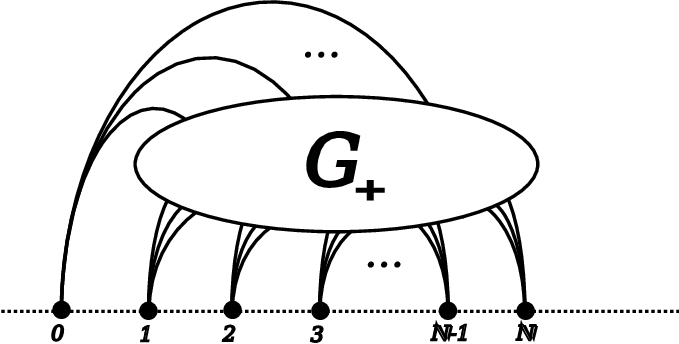}\end{array} \ar[r]^-{\sigma \cdot G_+} & \begin{array}{c} \includegraphics[width=2.3in]{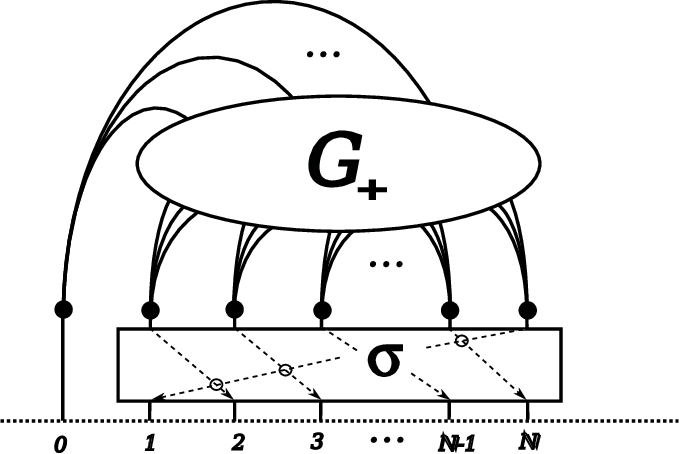}\end{array}} 
\]
\caption{Schematic depiction of the notation $\sigma \cdot G_+$ for $\sigma \in S_N$.} \label{fig_stacking_permuation}
\end{figure}

\begin{lemma} \label{lemma_realize_algo} If $G$ is Thompson, then there is a $\sigma \in S_N$ and a planar Thompson rooted tree $H_+$ such that $G_+=\sigma\cdot H_+$.\end{lemma}

\begin{proof} If $G$ has no virtual crossings, then there is nothing to do. Otherwise, let $v$ be the leftmost vertex having a rightwards edge with a virtual crossing. Let $e$ be the rightward pointing edge of $v$ having a virtual crossing whose right hand endpoint $v'$ is closest to $v$. Let $f$ be the edge of $G_+$ having a virtual crossing with $e$ whose right hand endpoint is closest to $v'$. Let $u$ be the left hand endpoint of $f$, where possibly $u=v$. Let $u'$ be the right hand endpoint of $f$. Between $v'$ and $u'$, there are some vertices $w_1,\ldots,w_k$. See Figure \ref{fig_real_proof_labels} for an illustration of this notation. Since $G$ is Thompson, each $w_i$ has exactly one leftwards edge $f_i$.

\begin{figure}[htb]
\includegraphics[width=3in]{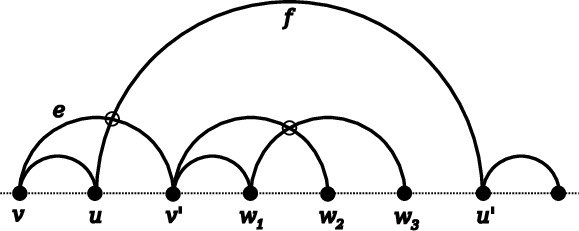}
\caption{Illustration of notation used in the proof of Lemma \ref{lemma_realize_algo}.} \label{fig_real_proof_labels}
\end{figure}
We first claim that $f_i$ must have its left endpoint $u_i$ at $v'$, after $v'$, or at $v$. If $u_i$ occurs between $v$ and $v'$, then $f_i$ must have a virtual crossing with $e$. But $w_i$ is closer to $v'$ than $u'$, so this contradicts our original choice of $f$. If $u_i$ occurs before $v$, then $f_i$ has a virtual crossing with the edge $f$. This contradicts the choice of $v$ as the leftmost vertex having a virtual crossing. This proves the claim.

Now, let $\sigma_1 \in S_N$ be the permutation of the vertices $1,\ldots,N$ that fixes all of them except $v',w_1,\ldots, w_k,u'$. On these vertices, the permutation is given by:
\[
v',w_1,\ldots, w_k,u' \to u',v',w_1,\ldots, w_k.
\]
Next we claim that $\sigma_1\cdot G_+$ is still Thompson. It must be shown that after the permutation, the leftwards pointing edge at each nonzero vertex still points left and that no rightwards pointing edge has turned leftwards. First consider $u'$. Its leftwards edge ends at $u$. The vertex $u$ must lie between $v$ and $v'$, with possibly $u=v$. Since $u \ne w_i$ for any $i$, $u$ still lies to the left of $u'$ after the permutation. Moreover, all of the right hand endpoints of rightwards edges at $u'$ are still to the right of $u'$. Secondly, consider the vertex $v'$.  Its leftwards edge $e$ still ends at $v$ on the left, so that the left edge is preserved. Rightwards edges at $v'$ can end at $w_1,\ldots,w_k$ or at vertices after $u'$. After applying $\sigma_1$, these vertices all remain after $v'$, so that the direction of the edges at $v'$ is preserved by $\sigma_1$. Lastly consider the vertices $w_i$, $1 \le i \le k$. The leftwards edge for $w_i$ has as its endpoint one of $v,v',\ldots, w_{i-1}$. Since $\sigma_1$ does not change the order of $v,v',w_1,\ldots,w_k$, the leftwards edge remains leftwards. Every rightwards edge at $w_i$ has its right hand endpoint at one of $w_{i+1},\ldots, w_k$ or a vertex after $u'$. In particular, this right hand endpoint cannot be $u'$, since this contradicts the fact that $u$ occurs before $v'$ and $u'$ has exactly one edge pointing left. Hence, every rightwards edge at $w_i$ remains rightwards after application of $\sigma_1$. Thus, $\sigma_1 \cdot G_+$ is Thompson, as claimed.  

Now, $\sigma_1 \cdot G_+$ may not have fewer virtual crossings than $G_+$. The number of virtual crossings may have even increased. However, the first problematic edge or vertex must occur further to the right. Indeed, if the vertex $v$ has no more edges with a virtual crossing, then the first problematic vertex is to the right of $v$. Otherwise, $v$ still has a rightwards edge with a virtual crossing. But this edge must have its right hand endpoint further away from $v$ than $v'$ was originally. Indeed, by hypothesis, there were no right hand endpoints of edges of $v$ with virtual crossings between $v$ and $v'$. Since $u'$ was moved to occur before $v'$, the next bad edge must be further right than before. 

Applying the entire process again to $\sigma_1 \cdot G_+$, we obtain a permutation $\sigma_2$ and a new Thompson graph $\sigma_2 \cdot (\sigma_1 \cdot G_+)$. After $s$ iterations, we have the Thompson graph $\sigma_s \cdot (\sigma_{s-1} \cdot \ldots\cdot(\sigma_2 \cdot (\sigma_1 \cdot G_+)))$. Since there are only finitely many vertices, the rightmost vertex has no rightwards edges, and each application of the process moves the next problematic edge or vertex to the right, there must be an integer $s$ such that $\sigma_s \cdot (\sigma_{s-1} \cdot \ldots\cdot(\sigma_2 \cdot (\sigma_1 \cdot G_+)))$ can be drawn as a planar graph. Then setting $H_+=\sigma_s \cdot (\sigma_{s-1} \cdot \ldots\cdot(\sigma_2 \cdot (\sigma_1 \cdot G_+)))$ and $\sigma= (\sigma_s\sigma_{s-1} \cdots \sigma_1)^{-1}$ proves the lemma.
\end{proof}

Now, putting all of the above together, we can finally prove Theorem \ref{thm_intro_real}.

\begin{theorem}\label{thm: realization} Every CC virtual link type $L$ is realized by an element of Thompson's group $V$.
\end{theorem}
\begin{proof} We summarize the argument above. The first step is to construct a virtual Tait graph $G$ for a checkerboard coloring $\xi$ of a representative of $L$ on surface $\Sigma$ of minimal genus. By Lemma \ref{lemma_prepare_standard}, $G$ is virtually 2-equivalent to a standard graph $G'$. By Lemma \ref{lemma_virtual_thompson}, $G'$ is virtually $2$-equivalent to a Thompson graph $G''=G_+'' \cup G_-''$. Let $N+1$ be the number of vertices of $G''$. The graph $G_-''$ has no virtual crossing and hence is planar. By Lemma \ref{lemma_realize_algo}, there is a $\sigma \in S_N$ such that $\sigma \cdot G_+''= H_+$, where $H_+$ is planar. Now, apply \cite{jones_unitary}, Lemma 4.1.4 to $H_+$ and $G_-''$ to obtain planar bifurcating trees $T_+$, $T_-$. Consider $S_{N+1}$ as the permutation group on the symbols $\{0,\ldots,N\}$ and view $\sigma \in S_{N+1}$ as the element which restricts to $\sigma \in S_{N}$ on the vertices $1,\ldots,N$ and fixes $0$. 

Then $\mathscr{L}_V(T_+,T_-,\sigma)$ is a diagram of the virtual link type of $L$. To see this, first note that $H_+=\sigma^{-1} \cdot G_+''$, where the permutation is drawn from top to bottom. The permutation drawn from bottom to top is $(\sigma^{-1})^{-1}=\sigma$, so that the correct permutation is indeed $\sigma$. Secondly, the algorithm in Lemma \ref{lemma_realize_algo} preserves the virtual link type of $G''$. Indeed, the cyclic ordering around each of the vertices in $G_+''$ is preserved at every step, as are the endpoints of each of the edges. Hence, each iteration is related to the previous one by detour moves and planar isotopies.  
\end{proof}

\subsection{Example: the vertical mirror image of 4.105} \label{sec_example} Here we illustrate the realization algorithm using the example begun in Section \ref{sec_virtual_links}. In terms of Green's table \cite{green}, it is the virtual knot obtained from 4.105 by changing all of its classical crossings $\oplus \leftrightarrow \ominus$. This is called the \emph{vertical mirror image} $\overline{K}$ of a virtual knot $K$. All of the classical crossings for the diagram $4.105$ in \cite{green} are negative, so our diagram for $\overline{4.105}$ has only positive classical crossings. A diagram of $\overline{4.105}$ on the torus is given in Figure \ref{fig_check_colored}. Let $\xi$ be the black checkerboard surface in the center of Figure \ref{fig_check_colored}. The virtual spanning surface and Tait graph $G$ are shown in Figures  \ref{fig_disc_band_example} and \ref{fig_virt_tait_graph_example}, respectively. 

The first task is to make $G$ standard, as described in Lemma \ref{lemma_prepare_standard}. Recall that this is done by subdividing the edges so that all the vertices can be placed on the non-negative real axis and so that the interior of each edge lies in either the first or fourth quadrant. One way of doing this is shown in Figure \ref{fig_prepare_virt_tait_4_105}. The graph labeled (i) is obtained from Figure \ref{fig_virt_tait_graph_example} by subdividing two edges. The labels on the vertices describe their positions on the $x$-axis in the subsequent picture. This convention is followed also in the graphs labeled (ii) and (iii). To simplify notation, we will continue to call this graph `$G$' after any modifications. Note that we are using more subdivisions than are strictly necessary to make $G$ standard. This reduces the number of edges signed $-\varepsilon$ in $G_{\varepsilon}$. Badly signed edges greatly increase the number vertices needed to make $G$ Thompson (see Lemma \ref{lemma_virtual_thompson}).

\begin{figure}[htb]
\[
\begin{array}{cc}
\begin{array}{c}
\includegraphics[width=1.3in]{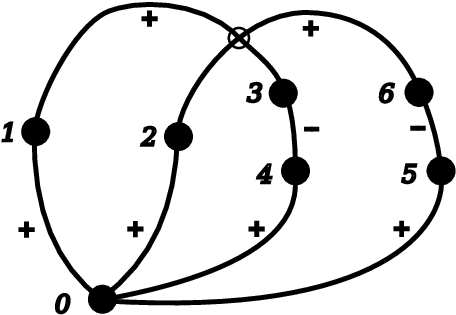}\\ \text{(i)}
\end{array} &\begin{array}{c}
\includegraphics[width=2.8in]{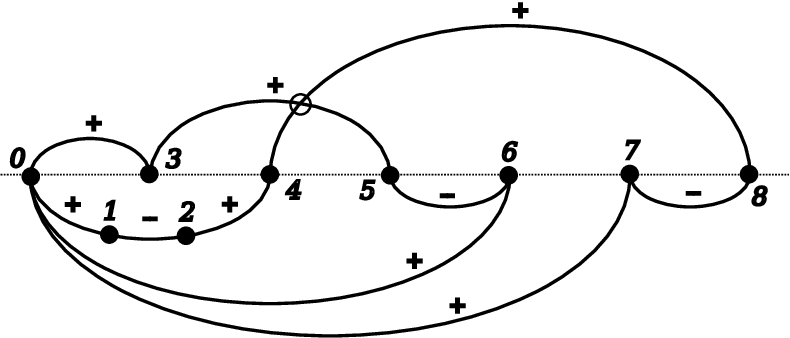} \\ \text{(ii)}
\end{array} \\
\begin{array}{c}
\includegraphics[width=2.8in]{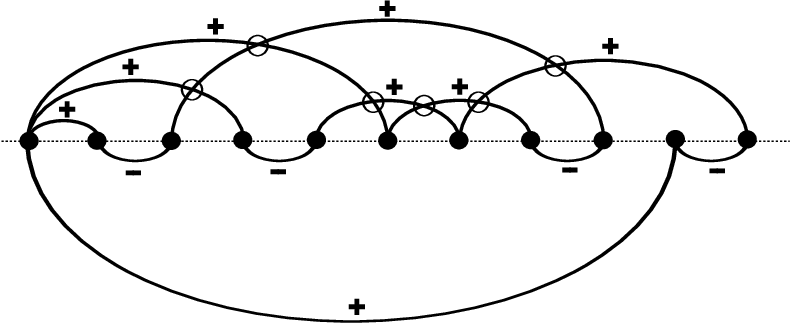} \\ \text{(iv)}
\end{array} & \begin{array}{c}
\includegraphics[width=2.8in]{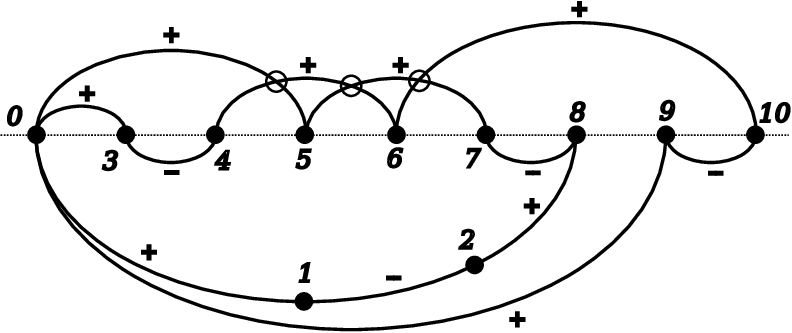} \\ \text{(iii)}
\end{array}
\end{array}
\]
\caption{Standardizing the virtual Tait graph of $\overline{4.105}$.}\label{fig_prepare_virt_tait_4_105}
\end{figure}

Another way to reduce the number of such bad edges is to exploit the non-planarity of virtual Tait graphs. Indeed, since edges represent half-twisted bands, they can be rerouted using detour moves as long as the cyclic ordering of vertices is preserved at each vertex. Since detour moves do not change the virtual link type, this will still produce the desired inverse image in $V$ of our original virtual link. We will employ this trick twice to reach our final standard, but still non-Thompson, graph. The first instance is going from (iii) to (iv) in Figure \ref{fig_prepare_virt_tait_4_105}. The edge between the vertices 2 and 8 is detoured into the first quadrant, and thereby adding three virtual crossings. Secondly, we will detour the edge between the vertices 0 and 9 in Figure \ref{fig_prepare_virt_tait_4_105} (iii). The result is shown in Figure \ref{fig_thompson_graph_4_105}, where we will momentarily ignore the red-colored hollow vertices and their incident edges. Their origin will be explained later. Note that the second detour has removed all $-\varepsilon$ signed edges in $G_{\varepsilon}$.

Now, the black subgraph in Figure \ref{fig_thompson_graph_4_105} is standard but not Thompson. There are no edges signed $-\varepsilon$ in $G_{\varepsilon}$ but there are vertices in which one of $G_{+},G_-$ does not have a leftwards edge. As in Lemma \ref{lemma_virtual_thompson}, this can be fixed using Type RI and Type RIIa moves (see Figure \ref{fig_type_moves}). The added hollow vertices and their incident edges are drawn in red in Figure \ref{fig_thompson_graph_4_105}. The graph $G$ is now Thompson. 

Next, we use Lemma \ref{lemma_realize_algo} to find a permutation $\sigma \in S_{18}$ such that $\sigma \cdot G_+=H_+$ is planar. The steps are shown in Figure \ref{fig_algo}. The algorithm initializes with $G_+$, as shown in Figure \ref{fig_algo} (i). Next, the algorithm cycles the vertices $6 \to 15$, which produces the graph in Figure \ref{fig_algo} (ii). We continue in this fashion, passing the right-hand endpoint of the leftmost non-planar edge over the right-hand endpoint of the nearest edge it crosses, together with all intermediate vertices between the two. Eventually this produces the desired planar graph $H_+$. The desired permutation $\sigma$ is the composition of the cycles. Lastly, we draw the planar bifurcating trees $T_+,T_-$ and the permutation $\sigma$, now viewed as an element of $S_{19}$ which fixes $0$. This is done in Figure \ref{fig_realize_4_105_done}. The graphs $H_+$ and $G_-$ are drawn as dashed semicircles and $\sigma$ is the composition of cycles in Figure \ref{fig_algo}. 

\begin{figure}[htb]
\includegraphics[width=4in]{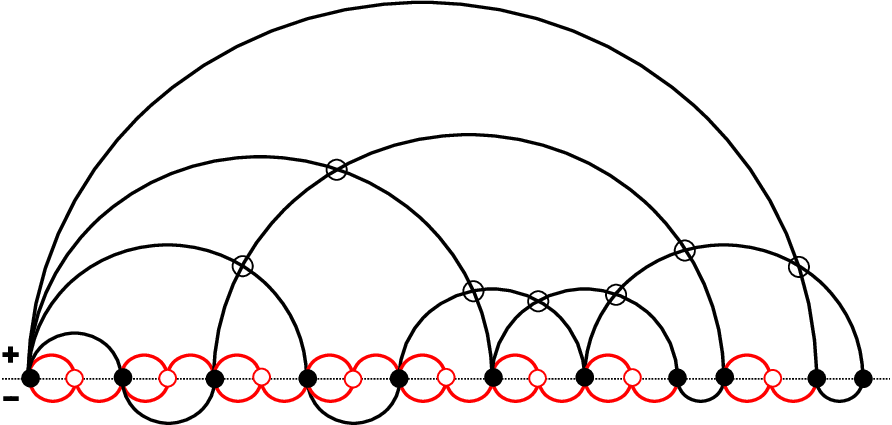}
    \caption{A Thompson graph for $\overline{4.105}$ which is 2-equivalent to the one in Figure \ref{fig_prepare_virt_tait_4_105}.}
    \label{fig_thompson_graph_4_105}
\end{figure}

\begin{figure}
    \[
    \begin{array}{ccc}
    \begin{array}{c}
    \includegraphics[width=1.75in]{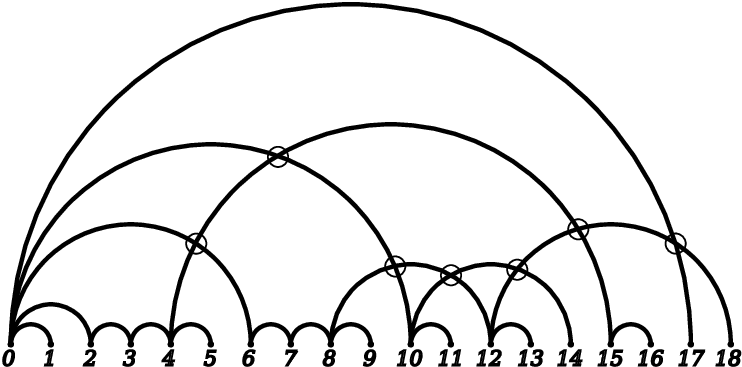} \\
    \text{(i) Cycle: } 6 \to 15
    \end{array} & \begin{array}{c}
    \includegraphics[width=1.75in]{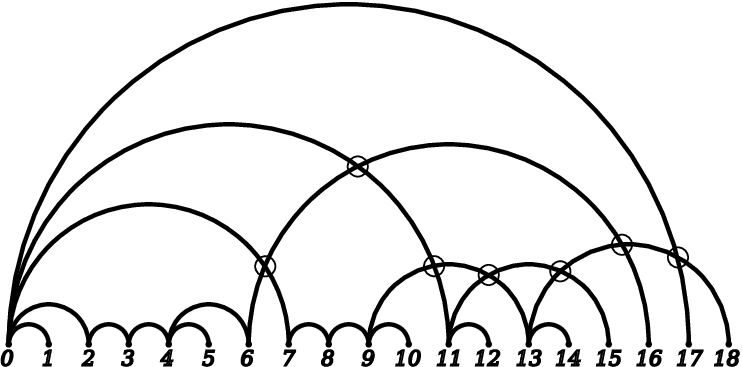}\\
    \text{(ii) Cycle: } 7 \to 16
    \end{array} & \begin{array}{c}
    \includegraphics[width=1.75in]{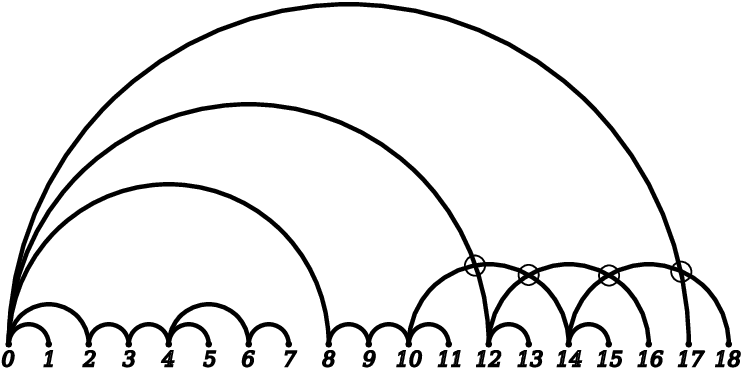}\\
    \text{(iii) Cycle: } 12 \to 14
    \end{array} \\
    \begin{array}{c}
    \includegraphics[width=1.75in]{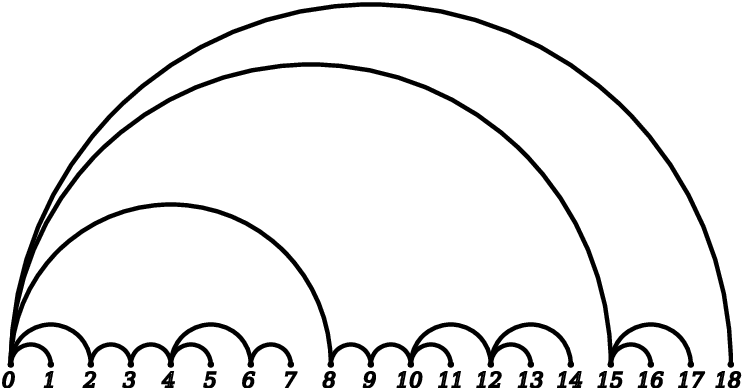}\\
    \text{(vi) Done! } 
    \end{array} & \begin{array}{c}
    \includegraphics[width=1.75in]{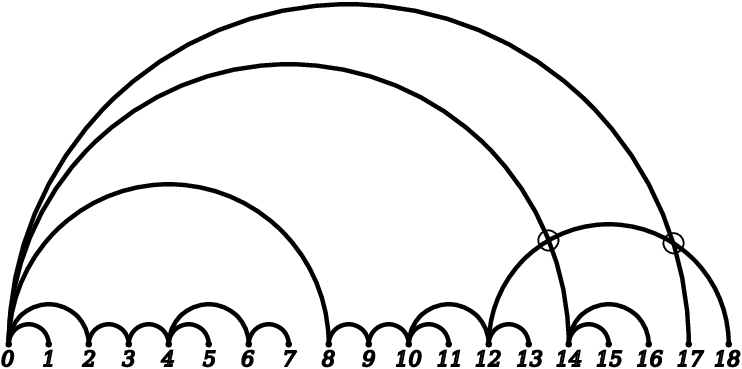}\\
    \text{(v) Cycle: } 14 \to 18
    \end{array} & \begin{array}{c}
    \includegraphics[width=1.75in]{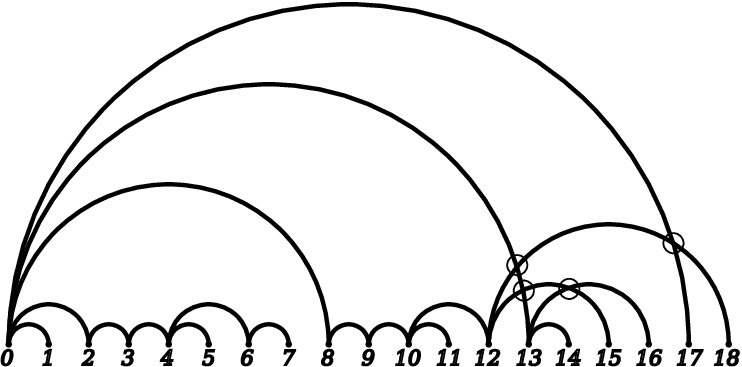}\\
    \text{(iv) Cycle: } 13 \to 15
    \end{array}
    \end{array}
    \]
    \caption{Starting with $G_+$ for $\overline{4.105}$ in Figure \ref{fig_thompson_graph_4_105} in (i), we perform the algorithm given in Lemma \ref{lemma_realize_algo}. The subsequent pictures (ii)-(vi) cycles the vertices as listed until a planar rooted tree $H_+$ is obtained.}
    \label{fig_algo}
\end{figure}
\begin{figure}[htb]
\includegraphics[width=4.5in]{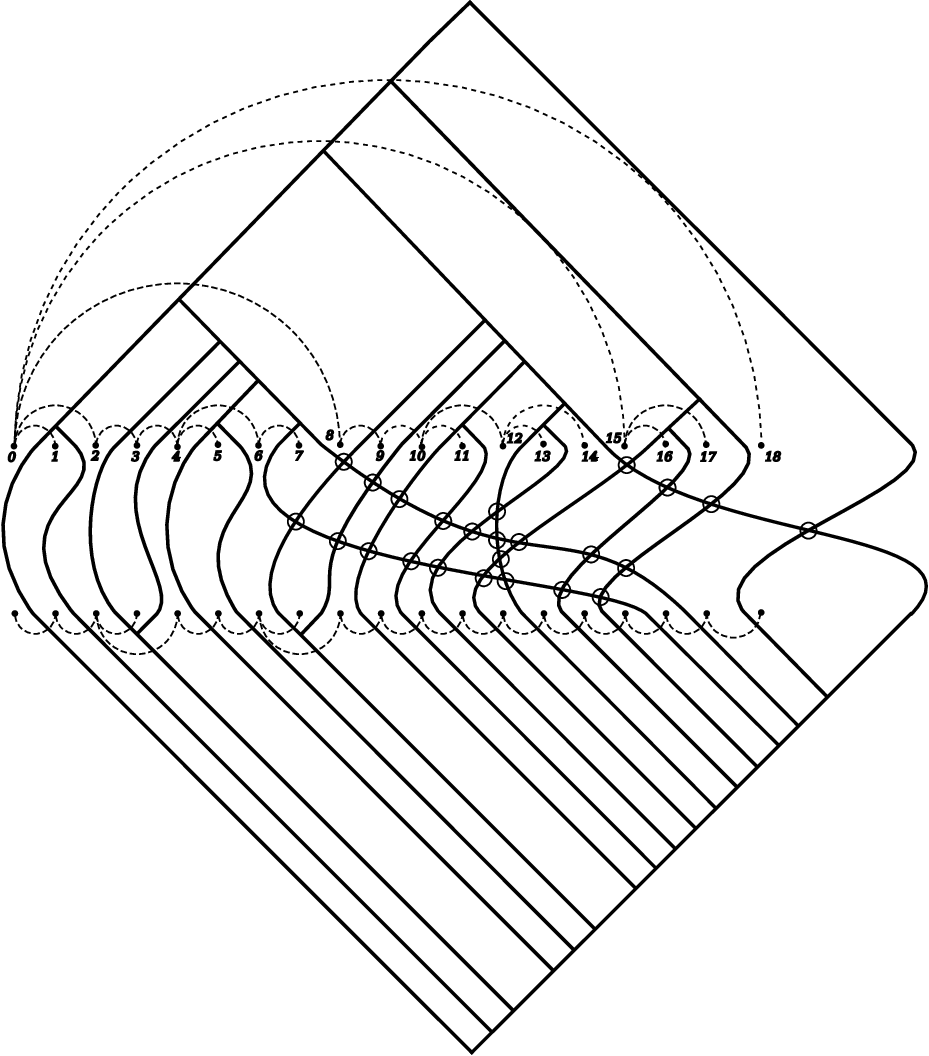}
    \caption{Realization of $\overline{4.105}$ by an element of $V$.}
    \label{fig_realize_4_105_done}
\end{figure}

\subsection{Example Comparison with $\mathit{VF}$} \label{sec_vf}
On the other hand, $\overline{4.105}$ can also be realized from an element of $\mathit{VF}$, following the procedure outlined in \cite{KTVF}. The group $\mathit{VF}$ is defined as a certain diagram group \cite[Definition 2.7]{KTVF}. Given $\Delta \in \mathit{VF}$, \cite{KTVF} describe how to build a medial graph $\Gamma(\Delta)$ from which a virtual link $L(\Gamma(\Delta))$ can be obtained. Before we define the medial graph for a virtual link, we caution the reader that although medial and Tait graphs are the same in the classical setting, they differ in the virtual setting.

Every virtual link diagram, even if it does not depict a CC virtual link, has a medial graph $\Gamma(D)$ such that $L(\Gamma(D))$ returns $D$. To construct $\Gamma(D)$, consider the shadow of $D$, which is the planar graph obtained from $D$ by removing the classical and virtual crossing data. The shadow can then be checkerboard shaded on the plane, as in the classical case, see Figure \ref{fig: VF tait graph}. Note that this checkerboard-shaded diagram does not always yield a checkerboard surface as described in Section \ref{sec_virt_tait_graphs}, because bands may have virtual half-twists. From this shaded diagram, we follow the usual convention to build a medial graph of a classical link diagram, except that edges originating from virtual crossings are given a label of $v$ rather than $+$ or $-$. This is where the medial graph differs from the virtual Tait graph--virtual Tait graphs encode virtual crossings as non-planar intersections of edges, whereas medial graphs encode virtual crossings as entirely new edges with new labels. In fact, $\Gamma(D)$ is always planar whereas its virtual Tait graph (if it exists) may not be.

Recall that when Jones constructed links from $F$, he associated a Tait graph (which is the same as a medial graph in the case of classical links) to a pair of planar, rooted, binary trees. Similarly, \cite{KTVF} show that one can also obtain a medial graph from a \textit{decorated pair of trees}, and then prove that all such medial graphs arise as $\Gamma(\Delta)$ for some $\Delta \in \mathit{VF}$ \cite[Lemmas 3.2 and 4.2]{KTVF}.

\begin{definition}
    A \textit{decorated pair of trees} $(T_+,T_-)$ is an ordered pair of planar, rooted, binary trees with the same number of leaves, in which some edges of both trees may be labeled with a $v$. We refer to these edges as ``decorated," and other edges as ``undecorated." From a decorated pair of trees one can create a medial graph as in \cite[Definition 4.1.2]{jones_unitary}, except that the edges of $\Gamma$ originating from decorated edges of $(T_+,T_-)$ are labeled with $v$ rather than $+$ or $-$.
\end{definition}

Kodama and Takano \cite{KTVF} prove that all virtual links arise from elements of $\mathit{VF}$ by showing that every virtual link $L$ admits a diagram $D$ whose medial graph $\Gamma(D)$ arises from a decorated pair of trees. To find such a pair of trees, one puts $\Gamma$ into standard position (as in Definition \ref{def_standard}, except condition (3) is irrelevant because the graph was planar to begin with.) The authors of \cite{KTVF} show that in order for a signed graph to have originated from an element of $\mathit{VF}$, it must satisfy all the conditions of Definition \ref{def_standard}, plus one more: two edges connecting the same pair of vertices must either both represent classical crossings (labeled $\pm$) or virtual crossings (labeled $v$). It cannot be the case that one edge has a $v$ and the other has a $\pm$. To correct this issue when it arises, the authors of \cite{KTVF} provide an additional step when, combined with the usual process of \cite{jones_unitary}, leads to a medial graph $\Gamma'$  which is $2$-equivalent to $\Gamma$ and originates from a decorated pair of trees. Figure \ref{fig: VF trees} shows $\Gamma,\Gamma'$, and the corresponding decorated pair of trees for $\overline{4.105}.$ 

    \begin{figure}[h!]\includegraphics[scale=0.5]{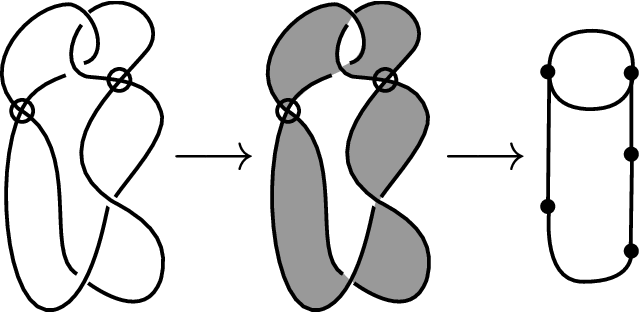}
    \put(-25,0){$+$}
    \put(-30,45){$v$}
    \put(-15, 68){$-$}
    \put(-15, 50){$-$}
    \put(0, 45){$+$}
    \put(0,20){$v$}
    \caption{The labeled graph associated to $\overline{4.105}$. Note the surface in the middle is not a checkerboard surface, and the top two crossings, despite being positive, inherit negative signs when this surface is translated into a decorated graph.}\label{fig: VF tait graph}\end{figure}

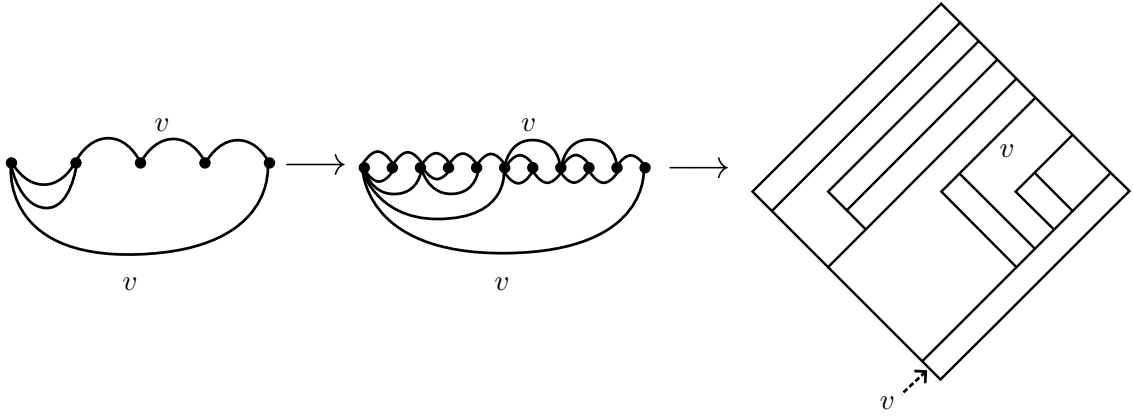
\begin{figure}\includegraphics[scale=0.6]{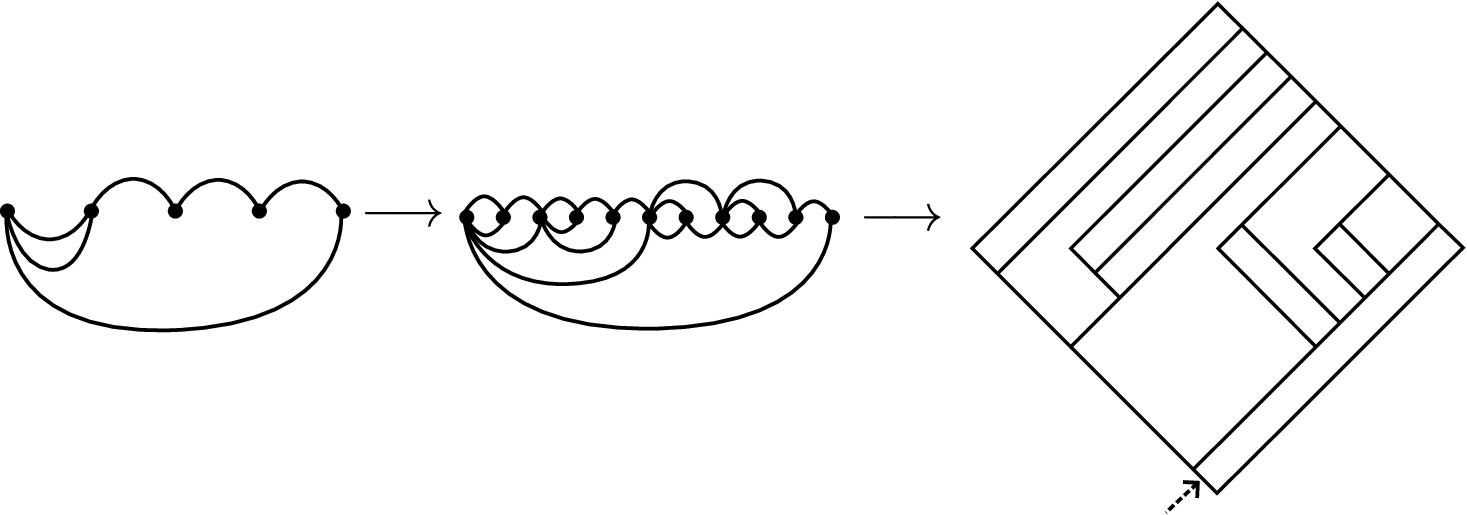}
\put(-380, 40){$v$}
\put(-368,100){$v$}
\put(-230, 100){$v$}
\put(-240,40){$v$}
\put(-95,-5){$v$}
%\put(-38, 90){$v$}
\put(-50, 90){$v$}
\caption{Applying the procedure of \cite{KTVF} to the labeled graph from Figure \ref{fig: VF tait graph} to obtain an element of $\mathit{VF}$ corresponding to $\overline{4.105}$. In this figure, all edges with maxima (resp. minima) that are \textit{not} labeled with a $v$ are understood to be positively (resp. negatively) signed.}\label{fig: VF trees}\end{figure}

\section{The Oriented Subgroup
$\vec{V}$}\label{sec: oriented} In Section \ref{sec_definition_vec_V}, we define the oriented subgroup $\vec{V}$, which contains Jones' oriented groups $\vec{F}\subset F$ and $\vec{T} \subset T$.  The proof of Theorem \ref{thm_intro_oriented}, that all AC links are in the image $\mathscr{L}_V(\vec{V})$, is in Section \ref{sec_realize_AC}. 
\subsection{Definition of $\vec{V}$} \label{sec_definition_vec_V} The first step of creating virtual links from elements of $V$ involved, for each of the trees in $(T_+,T_-, \sigma)$, building a planar tangle and checkerboard surface. Each tree's associated surface is always orientable, even though the checkerboard surface for $(T_+,T_-,\sigma)$ may not be after they are glued together according to $\sigma$. Following the convention that the leftmost band of a tree's associated surface is positively oriented, each tree induces a unique orientation on its associated surface. We may encode this orientation as a sequence of signs corresponding to the signs of the bands from left to right. These sequences, called $n$-signs, were introduced in \cite{jones_unitary} for the study of $\vec{F}$ and later used in \cite{aiellocontijones} to study unitary representations and the HOMFLYPT polynomial.

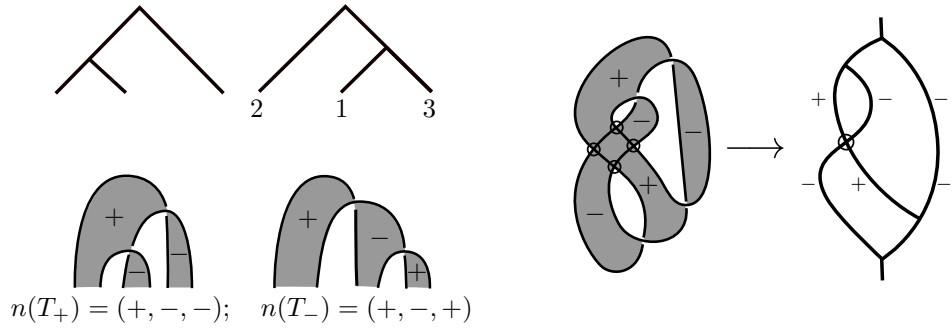
\begin{figure}
\small
    \includegraphics[scale=0.58]{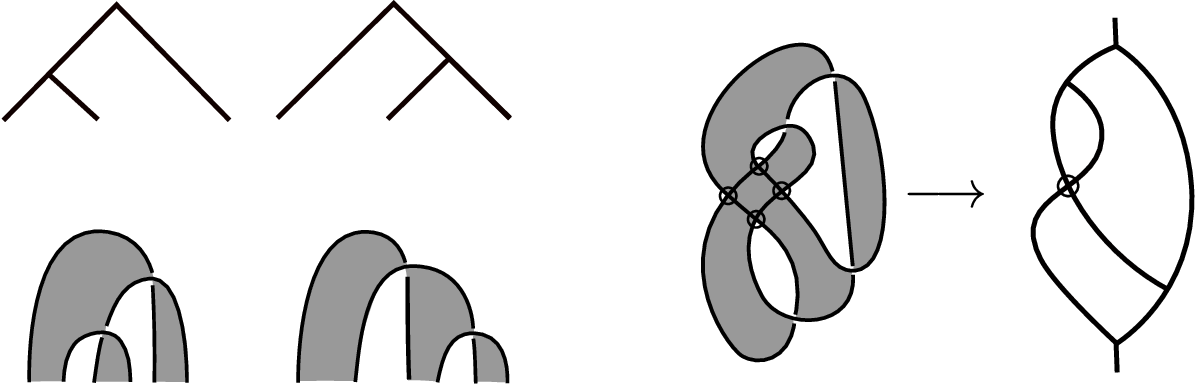}
    %permutation on left tree
    \put(-260, 65){$2$}
    \put(-228,65){$1$}
    \put(-195,65){$3$}
    %n-signs
    \put(-350,-10){$n(T_+)=(+,-,-);$}
    \put(-256,-10){$n(T_-)=(+,-,+)$}
    %signs on left semilink
    \put(-315, 25){$+$}
    \put(-306.5,4){\small $-$}
    \put(-290.5,11){\small $-$}
    %signs on right semilink
    \put(-242,24){$+$}
    \put(-215,17){\small $-$}
    \put(-201,4){\small $+$}
    %signs on shaded surface
    \put(-125,77){$+$}
    \put(-134,26){$-$}
    \put(-116.5,61){\small $-$}
    \put(-114,37){$+$}
    \put(-97, 57){$-$}
    %signs on strand diagram
    \tiny
    \put(-49,70){$+$}
    \put(-24,70){$-$}
    \put(-3,70){$-$}
    \put(-53,38){$-$}
    \put(-34,38){$+$}
    \put(-2,38){$-$}
    \normalsize
    \caption{ An element $v \in \vec{V}$, the $n$-signs associated to its trees, and the associated abstract strand diagram with leaves labeled. Note the signs of the second tree are switched upon gluing the bands because $n(T_{-})=-\sigma(n(T_+)).$}\label{fig: subgroup}
\end{figure}

\begin{definition}
    Let $S_{n}$ refer to the symmetric group on $n$ letters. Given an $n$-sign $\nu$ and a permutation $\sigma \in S_{n}$, define $\sigma(\nu)$ to be the image of the $n$-sign $\nu$ under the permutation $\sigma$. Let $-\nu$ refer to the $n$-sign resulting from switching all of the signs in $\nu$. 
\end{definition}

\begin{proposition}
    The triple $(T_+,T_-,\sigma)$ induces an orientable checkerboard surface if and only if $n(T_-))=\pm \sigma(n(T_+)).$
\end{proposition}

\begin{proof} Suppose $v=(T_+,T_-,\sigma) \in V$ is such that the diagram $\mathscr{L}(v)$ gives an orientable checkerboard surface. The sequence of signs on bands before the permutation is $n(T_+)$. The sequence of signs on bands after the permutation is $\sigma(n(T_+))$. If the leftmost band coming from the bottom tree is labeled with a $+$, then $\sigma(n(T_+))=n(T_-)$. If the leftmost band coming from the bottom tree is labeled with a $-$, then $\sigma(n(T_+))=-n(T_-)$.

    On the other hand, consider a triple $(T_+,T_-,\sigma)$ with $n(T_-)=\pm \sigma(n(T_+)).$ If $n(T_-)=\sigma(n(T_+))$, then $\sigma$  respects the orientation of the bands associated to the two trees, so they can be glued together to form an oriented surface. If $n(T_-)=-\sigma(n(T_+)),$ then reversing the orientation of $T_-$ and gluing it to $T_+$ via $\sigma$ results in an oriented surface.
\end{proof}

\begin{definition}\label{def: or}
    Let $\vec{V}$ refer to the subset of reduced triples $(T_+,T_-,\sigma) \in V$ satisfying \[\sigma(n(T_+))=\pm n(T_-).\]
\end{definition}

\begin{remark}\label{rem: signs}Given an element $v=(T_+,T_-,\sigma)\in \vec{V}$, one can use the compatible $n$-signs induced by the trees to label the associated abstract strand diagram, as in Figure \ref{fig: subgroup}. If $\sigma(n(T_+))=-n(T_-)$, then the leftmost leaf of $T_-$ will have a $-$ sign next to it in this diagram, rather than a $+$.\end{remark}

\begin{proposition}
    $\vec{V}$ is a subgroup of $V$. 
\end{proposition}
\begin{proof} Given $(T_+,T_-,\sigma) \in \vec{V}$, its inverse $(T_-,T_+,\sigma^{-1})$ is also in $\vec{V}$. The identity element, also in $\vec{V}$, has two trees with only one leaf and $\sigma=\text{id}$.

It remains to show $\vec{V}$ is closed. Consider $f,g\in \vec{V}$, $f=(T_+,T_-,\sigma),g=(S_+,S_-,\nu),$ and the abstract strand diagrams representing each element. We will show, using the diagrammatic composition procedure described earlier in Figure \ref{fig: comp in V}, that $g \circ f$ is given by a triple satisfying the condition in Definition \ref{def: or}.

In this process will work with signed versions of each diagram (as in Remark \ref{rem: signs}). By construction, the abstract strand diagrams for $g$ and $f$ have a top tree whose first leaf has a $+$ to its left, but the leftmost leaf of each bottom tree could have either a $+$ or a $-$. In order to stack $g$ below $f$ so that the signs are coherent, we may need to switch the orientation on $g$. 

We now keep track of what happens to the signs at each step of the diagrammatic composition process.
In Step (2) of this process, we reduce a tree stacked below an inverse tree, to an inverse forest stacked below a forest. Locally, we are working with oriented strand diagrams in the sense of \cite[Definition 4.1]{KLL}, and reduction preserves the orientability \cite[Definition 4.4]{KLL}. After performing step (3), the diagram consists of a pair of trees labeled with their $n$-signs (or, in the case of the bottom tree, perhaps the negative of its $n$-sign), and a permutation $\nu' \circ \sigma'$ that is compatible with the signs on the trees. The canceling carets of step (4) preserve orientation because Type I moves on oriented strand diagrams are compatible with orientations. Therefore $\vec{V}$ is closed and indeed is a subgroup.
\end{proof}

\subsection{Realizing AC virtual links} \label{sec_realize_AC} In \cite{aiello_oriented}, Aiello proved that every oriented classical link can be realized by some element of the oriented subgroup $\vec{F}$. In this subsection, we prove that every almost classical link can be realized by an element of $\vec{V}$. The proof is similar to the unoriented case given in Section \ref{sec_real}. Recall from Section \ref{sec_virt_tait_graphs} that every AC link type has an oriented virtual Tait graph, where both the edges and vertices carry a sign. See Figure \ref{fig_4_105_oriented_Tait_graph}. The main difference between the CC and AC realization algorithms is that the moves shown Figures \ref{fig_type_moves} and \ref{fig_subdivide_edge} are now oriented. In particular, any two adjacent vertices in an oriented virtual Tait graph must have opposite sign. This is due to the fact that an edge represents a half-twisted band and hence, the orientation must be opposite on each side of the twist. We now proceed with the proof of Theorem \ref{thm_intro_oriented}.   
\begin{figure}[htb]
\[
\xymatrix{\begin{array}{c}\includegraphics[width=1.4in]{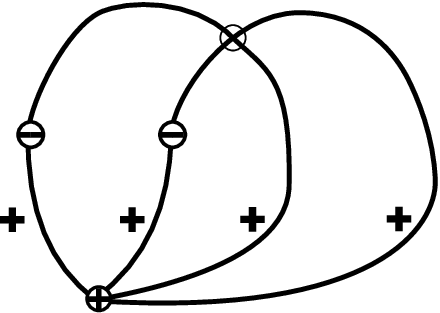}\end{array} \ar[r] & \begin{array}{c} \includegraphics[width=3.5in]{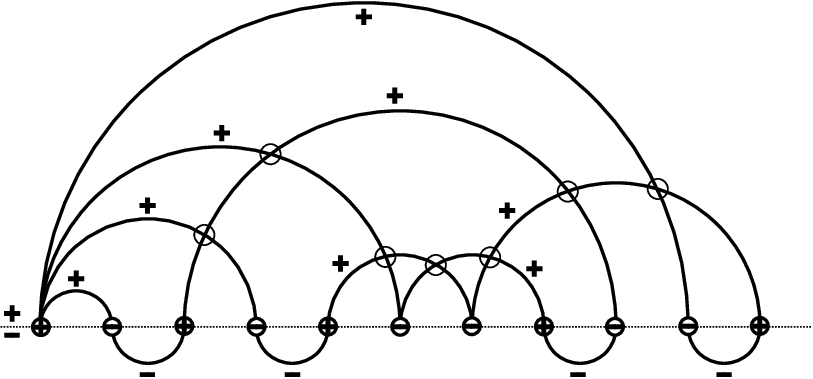}\end{array}}
\]
\caption{An oriented virtual Tait graph for $\overline{4.105}$. It is obtained from the oriented surface in Figure \ref{fig_check_colored}, center, by the virtual $2$-equivalence shown in Figure \ref{fig_prepare_virt_tait_4_105}.} \label{fig_4_105_oriented_Tait_graph}
\end{figure}

\begin{theorem}\label{thm: oriented realization} Every almost classical link type is realized by an element of $\vec{V}$.
\end{theorem}
\begin{proof} Let $L$ be an oriented almost classical link type. Then $L$ can be realized by a homologically trivial link $\mathcal{L}$ in a thickened surface $\Sigma \times [0,1]$, where $\Sigma$ is an oriented surface of minimal genus. As discussed in Section \ref{sec_virtual_links}, it may be assumed that $\Sigma$ is connected. By \cite{boden_chrisman_karimi}, Proposition 5.8, there is a diagram of $D$ of $\mathcal{L}$ on $\Sigma$ admitting a checkerboard surface $M$ which is connected and oriented. In particular, the orientation of $M$ induces the orientation of the link. If $\xi$ is the checkerboard coloring, we will assume that $M$ is the black surface. Let $\Gamma_{\xi}$ denote the oriented Tait graph for $M$ on $\Sigma$. Let $\Sigma_{\infty}$ be the compact surface obtained from deleting a small neighborhood of a point $\infty$ in one of white cells of $\xi$. For some handle decomposition of $\Sigma_{\infty}$, let $S \subset \mathbb{R}^2$ be the screen. Let $G$ be the oriented virtual Tait graph obtained from transferring $\Gamma_{\xi}$ to $S$. 

If $\Sigma=S^2$, so that $L$ is a classical link, the result follows as in \cite{aiello_oriented}. Otherwise we may assume $\Sigma$ has genus at least one. As in Section \ref{sec_real}, use planar isotopies, detour moves, and subdivisions of edges, to turn $G$ into a standard graph $G'=G'_+ \cup G'_- $(see Lemma \ref{lemma_prepare_standard}). Since the genus of $\Sigma$ is at least one, $L$ is not the unknot. Then since $M$ is oriented, $G$ must have at least two vertices and, in particular, at least one vertex that is signed $+$. Thus, we may assume that the leftmost vertex of $G'$ is signed $+$.  A standard graph for the oriented AC knot  $\overline{4.105}$ is shown in Figure \ref{fig_4_105_oriented_Tait_graph}. 

Next, it must be shown that $G'$ is virtually $2$-equivalent by a sequence of oriented Type RI, RIIa, and RIIb moves to an oriented Thompson graph. For classical oriented links, a sequence of such moves was given by Aiello in \cite{aiello_oriented}. These moves all occur in a small neighborhood of the $x$-axis. Hence, in the non-planar case, these may all be applied in the $0$-handle of the screen $S$, so that they do not involve any of the virtual crossings. This implies that $G'$ is virtually $2$-equivalent to an oriented Thompson virtual Tait graph $G''$. Suppose $G''$ has $N+1$ vertices. Since $G''$ is Thompson, we can apply Lemma \ref{lemma_realize_algo} to find a permutation $\sigma \in S_N$ such that $\sigma \cdot G''_+=H_+$ and $H_+$ is planar. Next, as in the proof of Theorem \ref{thm: realization}, apply \cite{jones_unitary}, Lemma 4.1.4, to the planar rooted trees  $(H_+,G''_-)$ to obtain a pair of planar rooted bifurcating trees $(T_+,T_-)$. Consider the element $(T_+,T_-,\sigma) \in V$ where $\sigma$ is now viewed as an element of $S_{N+1}$ that fixes the element $0$. By construction, $T_+$ and $T_-$ have the same $n$-signs. Hence, $(T_+,T_-,\sigma) \in \vec{V}$. Since the leftmost vertex of $G''$ is signed $+$, the orientation of the link induced by the Seifert surface agrees with the canonical orientation of $\mathscr{L}_V(T_+,T_-,\sigma)$, as defined in Section \ref{sec_definition_vec_V}. This completes the proof.  
\end{proof}

\section{Positive Definite Functions via operator quandle colorings}\label{sec: op colorings}

 The goal of this section is to show that operator quandle coloring invariants define positive definite functions on $\vec{V}$ and $V$. Section 5.1 reviews quandles, operator quandles, and their coloring invariants on virtual links. Section 5.2 discusses positive definite functions and proves Theorem \ref{intro-thm-operator-quandle}.

\subsection{The operator quandle coloring invariant}
We first give an overview of quandles and the operator quandle coloring invariant, starting with some preliminary definitions.

A \textit{quandle} $(Q,*)$ is a set $Q$ with a binary operation, $*$, satisfying the following axioms:
\begin{enumerate}
    \item for every $x\in Q$, $x*x=x$,
    \item for every $x \in Q$, the right translation map $y \mapsto y*x$ is invertible, and
    \item for all $x,y,z\in Q$, $(x*y)*z=(x*z)*(y*z)$.
\end{enumerate}
 Sets with this algebraic structure were independently discovered by many \cite{Tak43, Joy79, Mat82, Bri88}; however, the term was coined in Joyce's PhD thesis \cite{Joy79}. A \textit{kei} (or \textit{involutive quandle}) is a quandle $(Q,*)$ such that for all $x\in Q$, the right translation map is an involution, i.e., for all $x, y \in Q$, $(y*x)*x=y$. A quandle or kei is \textit{finite} whenever the underlying set is finite.
    
    Let $R_x(y)=y*x$ be the right translation map on a quandle $(Q,*)$. Then the \textit{right division} on $(Q,*)$ is defined to be $x/y:= R_x^{-1}(y)$.  Notice that when $(Q,*)$ is a kei, $x/y=x*y$ for all $x,y$. An \textit{endomorphism} of a quandle $(Q,*)$ is a mapping $f:Q\to Q$ such that for all $x,y\in Q$, $f(x*y)=f(x)*f(y)$. A bijective endomorphism on $(Q,*)$ is an \textit{automorphism}.
    
    In \cite{KMV25}, the authors introduce the following structure in order to define invariants of virtual links. (However, the idea of this structure goes back to Kauffman and Manturov \cite{KM05}.) An \textit{operator quandle} $(Q,*,\alpha)$ is a triple where $(Q,*)$ is a quandle and $\alpha$ is an automorphism of $(Q,*)$. When $|Q|$ is finite, we call $(Q,*,\alpha)$  a \textit{finite operator quandle}.

    \begin{remark}
        In \cite{KMV25}, the authors define operator quandles more generally to obtain a new invariant of multi-virtual links. In this case, an operator quandle is a triple $(Q,*,A)$, where $A$ is a list of pairwise commuting automorphisms of $(Q,*)$. As this paper is only concerned with virtual links, we just choose a single automorphism of the quandle for our definition.
    \end{remark}

Before defining operator quandle colorings and coloring invariants, we first establish some terminology. For the remainder of the paper, an \textit{arc} of a virtual link diagram $D$ is a portion of a strand of $D$ between an under-crossing or virtual crossing and the next under-crossing or virtual crossing. 
    
\begin{definition}
    Let $D$ be a virtual link diagram. 
    \begin{enumerate}
        \item Suppose $D$ is oriented. Let $(Q,*,\alpha)$ be an operator quandle. A \textit{$(Q,*,\alpha)$-coloring} of $D$ is an assignment of elements of $Q$ to the arcs of $D$ satisfying the rules at each crossing shown in Figure \ref{fig: operator_quandle_colorings_crossings}.
        \item Suppose $D$ is unoriented. If $(Q,*)$ is a kei, then a $(Q,*,\text{id})$\textit{-coloring} of $D$ is an assignment of elements of $Q$ to the arcs of $D$ satisfying the unoriented rules in Figure \ref{fig: operator_quandle_colorings_crossings} where $\alpha=\text{id}$. 
    \end{enumerate}
\end{definition}

\begin{figure}[h!]
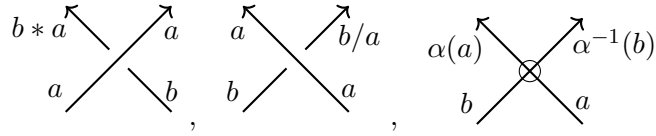

    \centering
    \operatorquandlecoloringpositive,  \operatorquandlecoloringnegative, \operatorquandlecoloringvirtual
    \caption{Operator quandle colorings at each type of crossing on a virtual link diagram.}
    \label{fig: operator_quandle_colorings_crossings}
\end{figure}

 Let $D$ be an oriented virtual link diagram and $(Q,*,\alpha)$ be an operator quandle.  Denote the number of $(Q,*,\alpha)$-colorings of $D$ as $\text{Col}(D, (Q, *, \alpha))$. Similarly, if $D$ is an unoriented virtual link diagram and $(Q,*)$ a kei, denote $\text{Col}(D,(Q,*,\text{id}))$ as the number of $(Q,*,\text{id})$-colorings of $D$. The number $\text{Col}(D, (Q, *, \alpha))$ (respectively, $\text{Col}(D,(Q,*,\text{id}))$) was shown to be an invariant of oriented (respectively, unoriented) virtual links in \cite{KM05} (Section 2.4).

\subsection{Positive definite functions on $V$ and $\vec{V}$} \label{sec_proof_of_C}
Here we define positive definite functions on $V$ and $\vec{V}$ from operator quandle colorings. The subsection concludes with the proof of Theorem \ref{intro-thm-operator-quandle}. 

Recall that different triples $(T_+,T_-,\sigma)$ representing the same element of $V$ differ by the addition of canceling carets. Next we establish how operator quandle colorings behave under these additions. 

\begin{lemma}\label{lem: numberofopqcoloringswithcarets}
    Consider a triple $(T_+, T_-, \sigma)$ of two trees with $n$ leaves and an element $\sigma \in S_n$ representing an element $v \in \vec{V}$. Let $(S_+, S_-, \tau)$ be the new triple of trees with $n+1$ leaves and $\tau \in S_{n+1}$ obtained by introducing a canceling caret. Then
    \begin{align*}
        \text{Col}(\mathscr{L}_{V}(S_+,S_-,\tau), (Q, *,\alpha))=|Q|\text{Col}(\mathscr{L}_{V}(T_+,T_-,\sigma), (Q,*,\alpha)),
    \end{align*}
    for any finite operator quandle $(Q,*,\alpha)$.
\end{lemma}

\begin{proof}
    The result follows from Lemma \ref{lem: cancelingcarets-resultsinunknot}, since the unknot can be colored in $|Q|$  ways. \qedhere
\end{proof}

\begin{definition}\label{def: oq-positive-type}
    Fix a finite operator quandle $(Q,*,\alpha)$. The $(Q,*,\alpha)$-\textit{coloring function} on $\vec{V}$, $\text{Col}(\cdot, (Q,*,\alpha)):\vec{V}\to \mathbb{R}_{\geq 0}$, is defined as
    \begin{align*}
        \text{Col}(v, (Q,*,\alpha)) := |Q|^{-n}\text{Col}(\mathscr{L}_{V}(T_+, T_-,\sigma), (Q,*,\alpha)),
    \end{align*}
    where $v \in \vec{V}$ is represented by the triple $(T_+, T_-,\sigma)$, $T_{\pm}$ have $n$ leaves, and $\sigma\in S_n$. 
\end{definition}

\begin{remark}\label{rmk: well-defined-positive-function}
    By Lemma \ref{lem: numberofopqcoloringswithcarets}, the addition of the normalization factor of $|Q|^{-n}$ in Definition \ref{def: oq-positive-type} makes the $(Q,*,\alpha)$-coloring function well-defined up to canceling carets. 
\end{remark}

We now rewrite the coloring invariant as a partition function on the underlying 4-valent graph of the virtual link diagram. This formulation of the function is better suited to proving it is positive definite because, as we will see, the coloring count will be decomposed into separate contributions from the upper and lower parts of the virtual link. Fix a finite operator quandle $(Q,*,\alpha)$. Let $D$ be an oriented virtual link diagram. Let $c_+$, $c_-$, and $c_{\text{o}}$ be a classical positive, classical negative, and virtual crossing, respectively. We view $D$ as a 4-valent graph, $\Gamma$, where the vertices are the classical and virtual crossings of $D$. Consider any vertex that was a classical crossing, $c_+$ or $c_-$, in the virtual link diagram. Denote $a_{\text{in}}^{\text{over}}, a_{\text{out}}^{\text{over}}$ to be the two edges that were the over-strand of the crossing and $b_{\text{in}}^{\text{under}}, b_{\text{out}}^{\text{under}}$ to be the two edges that were the under-stands of the crossing. Now, for a virtual crossing $c_{\text{o}}$ of $D$, set $a_{\text{in}}, a_{\text{out}}, b_{\text{in}}, b_{\text{out}}$ to be as in Figure \ref{fig: virtual crossings for partition function}.

\begin{figure}[h!]
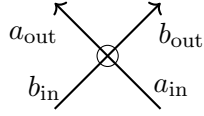

    \centering
    \operatorquandlevirtualgraph
    \caption{Labels for the four strands of a virtual crossing for a graph $\Gamma$.}
      \label{fig: virtual crossings for partition function}
\end{figure}

For a positive crossing, $c_+$, define:
\begin{align*}
    \omega_{c_+}(a_{\text{in}}^\text{over}, b_{\text{in}}^\text{under}, a_{\text{out}}^\text{over},b_{\text{out}}^\text{under}):= \begin{cases}
        1 &\text{ if } a_\text{in}^\text{over}=a_\text{out}^\text{over}, b_\text{out}^\text{under}=b_\text{in}^\text{under}*a_\text{in}^\text{over},\\
        0 &\text{ otherwise.}
    \end{cases}
\end{align*}
For a negative crossing, $c_-$, set:
\begin{align*}
    \omega_{c_-}(a_{\text{in}}^\text{over}, b_{\text{in}}^\text{under}, a_{\text{out}}^\text{over},b_{\text{out}}^\text{under}):=\begin{cases}
        1 &\text{ if } a_{\text{in}}^\text{over}=a_\text{out}^\text{over}, b_\text{out}^\text{under}=b_{\text{in}}^\text{under}/a_\text{in}^\text{over}, \\
        0 &\text{ otherwise.}
    \end{cases}
\end{align*}
Lastly for a virtual crossing, $c_{\text{o}}$, define:
\begin{align*}
    \omega_{c_\text{o}}(a_\text{in}, b_\text{in}, a_\text{out}, b_\text{out}):=\begin{cases}
        1&\text{ if } a_{\text{out}}=\alpha(a_\text{in}), b_\text{out}=\alpha^{-1}(b_\text{in}),\\
        0&\text{ otherwise.}
    \end{cases}
\end{align*}

Consider the partition function on $\Gamma$ defined by:

\begin{align*}
    Z_\Gamma(Q,*,\alpha):= \sum_{\substack{\tau\\\tau: E(\Gamma)\to Q}}\left(\prod_{\substack{c \in \mathcal{C}_+\\ c\in \mathcal{C}_-} }\omega_{c} (a_{\text{in}}^\text{over}, b_{\text{in}}^\text{under}, a_{\text{out}}^\text{over},b_{\text{out}}^\text{under})\prod_{c_\text{o}\in \mathcal{C}_\text{o}}\omega_{c_\text{o}}(a_\text{in}, b_\text{in}, a_\text{out}, b_\text{out})\right),
\end{align*}

\noindent where $\mathcal{C}_+,\mathcal{C}_-,\mathcal{C}_\text{o}$ are the sets of positive, negative, and virtual crossings of $D$ respectively, and the summation is over all functions $\tau$ on the set $E(\Gamma)$ of edges of $\Gamma$. 

\begin{remark}\label{rem: opq_colorings_is_Z}
    For $v \in \vec{V}$, the corresponding oriented virtual link diagram $\mathscr{L}_V(T_+,T_-,\sigma)$ satisfies: 
    \begin{align*}
        \text{Col}(\mathscr{L}_{V}(T_+,T_-,\sigma),(Q,*,\alpha))=Z_\Gamma(Q,*,\alpha),
    \end{align*}
    where $\Gamma$ is the 4-valent graph obtained from $\mathscr{L}_{V}(T_+,T_-,\sigma)$. 
\end{remark}

\begin{definition}
    Any virtual link diagram $\mathscr{L}_{V}(T_+, T_-,\sigma)$ built from an element $v \in V$ can be decomposed into a top and bottom half denoted $\mathscr{L}(T_+,\sigma)$ and $\mathscr{L}(T_-, \sigma)$ respectively, where we choose all the virtual crossings to lie entirely in the top half. We call these halves \textit{virtual semi-links}, see Figure \ref{fig:virtual_semilinks}. The strands of $\mathscr{L}_{V}(T_\pm,\sigma)$ outside of the box form the \textit{middle} of the two virtual semi-link diagrams when glued together.
\end{definition}

\begin{remark}
    By definition, the middle of the two virtual semi-link diagrams when glued together has no crossings. It will be $2n$ straight strands, where $n$ is the number of leaves in $T_{\pm}$.
\end{remark}

\begin{figure}[h!]
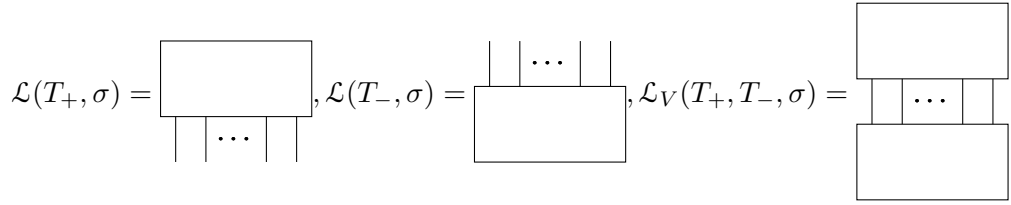

    \centering
    $ \mathscr{L}(T_+,\sigma)=\virtualsemilinkdown, \mathscr{L}(T_-,\sigma)=\virtualsemilinkup, \mathscr{L}_V(T_+, T_-, \sigma)=\multipliedvirtualsemilinks$
    \caption{Virtual semi-links and their multiplication to become the entire link $\mathscr{L}_V(T_+, T_-, \sigma).$}
    \label{fig:virtual_semilinks}
\end{figure}

Before proving that the $(Q,*,\alpha)$-coloring function on $\vec{V}$ is positive definite, we recall a few definitions. A continuous function on  group $G$, $\phi:G \to \mathbb{C}$, is said to be \textit{positive definite} (or \textit{of positive type}) if for any $r$ and $g_1,...,g_r\in G$, the matrix $(\phi(g_ig_j^{-1}))_{1\leq i,j \leq r}$ is positive semidefinite. Since Thompson's groups $F\subset T \subset V$ are all discrete, any function from one of Thompson's groups to $\mathbb{C}$ is continuous. We use the following lemma to show that a matrix is positive semidefinite.

\begin{lemma}[Adapted from \cite{Bha07}, p. 2-3]
    A matrix $A\in M_r(\mathbb{C})$ is positive semidefinite if and only if there exists an $r$-dimensional Hilbert space $H$ and vectors $\xi_1,...,\xi_r\in H$ such that
    \begin{align*}
        (A_{ij})_{1\leq i,j \leq r}=(\langle \xi_i,\xi_j\rangle)_{1\leq i,j \leq r},
    \end{align*}
    where $\langle \cdot, \cdot\rangle$ denotes the inner product on $H$.
\end{lemma}

    From a positive definite function $\phi$ on a discrete group $G$, the Gelfand-Naimark-Segal (GNS) construction produces a unitary representation  $\pi_\phi: G \to \mathcal{U}(K)$, where $K$ is a Hilbert subspace of the complex group algebra $\mathbb{C}[G]$ whose inner product is built from $\phi$, and $\mathcal{U}(K)$ is the group of unitary operators on $K$. See, for example, \cite{EW2025} Proposition 1.72. Conversely, positive definite functions always arise as diagonal matrix coefficients of unitary representations (\cite{EW2025}, Proposition 1.71).

Finally, we can prove part (2) of Theorem \ref{intro-thm-operator-quandle}. The proof is similar to \cite{AC19}, Theorem 6.3\footnote{There appears to be a misprint in their proof: the Hilbert space $H$ should be $\mathbb{C}^{|Q|^{2n}}$, not $\mathbb{C}^{3^{2n}}$.}.

\begin{theorem}\label{thm: function_of_positive_type_oriented}
    The function $\text{Col}(\cdot,(Q,*,
\alpha)):\vec{V}\to \mathbb{R}_{\geq 0}$, where $(Q,*,\alpha)$ is any finite operator quandle is positive definite on $\vec{V}$.
\end{theorem}

\begin{proof}
    Fix a finite operator quandle $(Q,*,\alpha)$. Choose $r\in \mathbb{N}$ and $v_1,...,v_r \in \vec{V}$. We need to show that $(\text{Col}(v_iv_j^{-1}, (Q,*,\alpha)))_{1\leq i,j \leq r}$ is positive semidefinite. For any $1\leq i,j\leq r$, define $v_i:=\mathscr{L}_{V}(T_+^i,T_-^i, \sigma_i)$ and $v_j:=\mathscr{L}_{V}(T_+^j, T_-^j, \sigma_j)$ to be the virtual link diagrams representing $v_i$ and $v_j$ respectively where $T_+^i, T_-^i, T_+^j, T_-^j$ are four trees with $n$ leaves and $\sigma_i, \sigma_j \in S_n$, where $n$ is large enough to satisfy that for all $1\leq i,j \leq r$, $T_-^i=T_-^j$. Then $v_iv_j^{-1}=\mathscr{L}_{V}(T_+^i, T_+^j, \sigma_j^{-1} \sigma_i)$. By Remark \ref{rmk: well-defined-positive-function}, it suffices to show $\left(\text{Col}(\mathscr{L}_{V}(T_+^i, T_+^j, \sigma_j^{-1}\sigma_i), (Q,*,\alpha))\right)_{1\leq i,j\leq r}$ is positive semidefinite. 

    Let $\Gamma_{i,j}$ be the 4-valent graph created from the virtual link diagram $\mathscr{L}_{V}(T_+^i, T_+^j, \sigma_j^{-i}\sigma_i)$. We can decompose any function $\tau:E(\Gamma_{i,j})\to Q$ as $\tau=(\tau_0, \tau_+, \tau_-)$ where $\tau_0$ is the function on the middle $2n$ edges, and $\tau_\pm$ are the functions on the edges of the virtual semi-links (except the middle). Using Remark \ref{rem: opq_colorings_is_Z}, we can rewrite each entry of our matrix as:

    \begin{align*}
        &\text{Col}(\mathscr{L}_{V}(T_+^i,T_+^j, \sigma_j^{-1}\sigma_i, (Q,*,\alpha))=Z_\Gamma(Q,*,\alpha)\\
        &= \sum_{\tau_0} \Biggl(\sum_{\tau_+}\Biggl(\prod_{\substack{c \in \mathcal{C}_+^i\\ c\in \mathcal{C}_-^i} }\omega_{c} (a_{\text{in}}^\text{over}, b_{\text{in}}^\text{under}, a_{\text{out}}^\text{over},b_{\text{out}}^\text{under})\prod_{c_\text{o}\in \mathcal{C}_\text{o}^i}\omega_{c_\text{o}}(a_\text{in}, b_\text{in}, a_\text{out}, b_\text{out})\Biggr)\\
        &\times \sum_{\tau_-}\Biggl(\prod_{\substack{c \in \mathcal{C}_+^j\\ c\in \mathcal{C}_-^j} }\omega_{c} (a_{\text{in}}^\text{over}, b_{\text{in}}^\text{under}, a_{\text{out}}^\text{over},b_{\text{out}}^\text{under})\prod_{c_\text{o}\in \mathcal{C}_\text{o}^j}\omega_{c_\text{o}}(a_\text{in}, b_\text{in}, a_\text{out}, b_\text{out})\Biggr)\Biggr).
    \end{align*}
    The first sum, over $\tau_0$, is over the $|Q|^{2n}$ functions coloring the $2n$ strands in the middle of the two virtual semi-links. The sums over $\tau_+, \tau_-$ are over the functions coloring the strands in the respective upper and lower virtual semi-link. The sets $\mathcal{C}_\pm^i, \mathcal{C}_o^i$ now only contain the crossings from the $T_+^i$ tree (likewise for $j$). 

    Since $\tau_0$ is a function on each of the $2n$ strands in the middle, we can write $\tau_0=(\tau_1,...,\tau_{2n})\in |Q|^{2n}$. Recall that $\ell^2(|Q|^{2n})=\{f:|Q|^{2n} \to \mathbb{C}:\sum_{\tau_0\in |Q|^{2n}} |f(\tau_0)|^2<\infty\}$ is a Hilbert space with inner product $\langle f, g\rangle =\sum_{\tau_0 \in |Q|^{2n}} f \overline{g}$. This Hilbert space is isomorphic to $\mathbb{C}^{|Q|^{2n}}$. We claim that $f_i:|Q|^{2n}\to \mathbb{C}$ defined by:
    \begin{align*}
        \tau_0 \mapsto \sum_{\tau_+}\left(\prod_{\substack{c \in \mathcal{C}_+^i\\ c\in \mathcal{C}_-^i} }\omega_{c} (a_{\text{in}}^\text{over}, b_{\text{in}}^\text{under}, a_{\text{out}}^\text{over},b_{\text{out}}^\text{under})\prod_{c_\text{o}\in \mathcal{C}_\text{o}^i}\omega_{c_\text{o}}(a_\text{in}, b_\text{in}, a_\text{out}, b_\text{out})\right),
    \end{align*}
    is in $\ell^2(|Q|^{2n})$. Indeed, $\sum_{\tau_0\in |Q|^{2n}}|f_i(\tau_0)|^2=\sum_{\tau_0\in |Q|^{2n}}f_i(\tau_0)^2$ is finite since the sum itself is finite. Further, for any $1\leq i,j \leq r$, $\langle f_i, f_j\rangle=Z_\Gamma(Q,*,\alpha)=\text{Col}(\mathscr{L}_{V}(T_+^i,T_+^j, \sigma_j^{-1}\sigma_i, (Q,*,\alpha))$, as we wished. By taking an $r$-dimensional closed subspace containing $f_1,...,f_r$ in $\ell^2(|Q|^{2n})$, we conclude that the matrix $\left(\text{Col}(\mathscr{L}_{V}(T_+^i, T_+^j, \sigma_j^{-1}\sigma_i), (Q,*,\alpha))\right)_{1\leq i,j \leq r}$ is positive semidefinite and further that the function $\text{Col}(\cdot,(Q,*,
\alpha))$ is positive definite on $\vec{V}$.
\end{proof}

By restricting to kei and requiring $\alpha$ to be the identity quandle automorphism, we can extend the above result to all of $V$. The below corollary proves Theorem \ref{intro-thm-operator-quandle} part (1).

\begin{corollary}
    Let $(Q,*)$ be a finite kei. Then the function $\text{Col}(\cdot, (Q,*,\textup{id})):V \to \mathbb{R}_{\geq 0}$ is a positive definite function on $V$.
\end{corollary}

\begin{proof}
   It is well known that kei colorings are unoriented virtual link invariants. Letting $\alpha$ be the identity, the same proof as Theorem \ref{thm: function_of_positive_type_oriented} gives the result.
\end{proof}

\section{Future Directions and Open Questions}\label{sec: openq}
Our results open up several avenues for future study; here we list some of them. In Section \ref{sec: oriented}, we defined an oriented subgroup $\vec{V} < V$. By \cite{KTVF}, there is also an oriented subgroup $\vec{\mathit{VF}} < \mathit{VF}$. 
\begin{question} What is the relationship between $\mathit{VF}$ and $V$ (and their oriented subgroups)?
\end{question} 
After Jones introduced $\vec{F}$ and $\vec{T}$, their algebraic properties were further studied by \cite{golansapir, nikkelren}. Golan and Sapir \cite{golansapir} gave a finite presentation for $\vec{F}$ and an isomorphism $\vec{F}\cong F_{3}$, the ternary Brown-Thompson group whose elements are encoded by pairs of ternary trees rather than binary trees. Furthermore, since $\vec{F}<F$ coincides with its commensurator, the unitary representations of $\vec{F}$ are irreducible. Nikkel and Ren \cite{nikkelren} showed that $\vec{T}$ coincides with its commensurator, but unlike $F$, $T\not\cong T_3$, however $\vec{T}$ and $T_3$ admit similar presentations as \textit{annular diagram groups.} 
\begin{question}
    Does $\vec{V}$ (resp. $\vec{\mathit{VF}}$) coincide with its commensurator in $V$ (resp $\mathit{VF}$)?
\end{question}

\begin{question} What is the relationship between $\vec{V}$ and $V_3$?
\end{question}

Jones observed that the realization theorem for classical links by elements of $F$ can be interpreted as an analog for Alexander's theorem for braids \cite{jones_unitary}. That is, every link in $S^3$ appears as the closure of some braid. In \cite{jones18}, Jones asked if there was a Markov-like theorem for Thompson's groups. The Alexander and Markov theorems for virtual braids were first proved by S. Kamada \cite{s_kamada_07}. Jones' question then generalizes to:

\begin{question} Is there are Markov-type theorem for $V$? In other words, what are necessary and sufficient conditions for $v_1,v_2 \in V$ such that $\mathscr{L}_V(v_1)$ and $\mathscr{L}_V(v_2)$ are equivalent virtual links?  
\end{question}

The \emph{Thompson $F$-index} $\text{ind}_F(L)$ of a classical link $L$ is the smallest number of leaves needed for $L$ to be realized by an element of $F$ \cite{jones_unitary,jones_no_go}. One can similarly define $\text{ind}_V(L)$ as the smallest number of leaves for a CC virtual link to be realized by an element of $V$. If $L$ is classical, $\text{ind}_V(L) \le \text{ind}_F(L)$. However, virtually equivalent classical links are classically equivalent as well.

\begin{question} If $L$ is classical link type in $S^3$, is $\text{ind}_V(L)= \text{ind}_F(L)$?
\end{question}

As is well known, the Artin braid groups are torsion-free. Both $V$ and $T$, however, contain elements of finite order. In fact, $V$ contains a copy of every finite group. What can be said about virtual links obtained from elements of finite order in $V$? In particular:
\begin{question} Are there elements $v \in V$ of finite order such that $\mathscr{L}_V(v)$ is non-trivial?
\end{question}

\bibliographystyle{alpha}
\bibliography{bib}

\end{document}